\documentclass[11pt]{amsart}

\usepackage[all]{xy}
\usepackage{mathrsfs}
\usepackage{amsmath, amsthm, amssymb, amscd, amsgen,amsxtra,amsfonts, amsbsy}
\allowdisplaybreaks
\usepackage[all]{xy}
\usepackage{verbatim}

\usepackage{tikz}
\usepackage{pgfplots}
\usepackage{tikz-cd}
\usepackage{etex}
\usepackage{calligra,mathrsfs}
\usepackage{color}
\definecolor{shadecolor}{gray}{0.875}
\definecolor{dblue}{rgb}{0,0,.6}
\usepackage[colorlinks=true, linkcolor=dblue, citecolor=dblue, filecolor = dblue, menucolor = dblue, urlcolor = dblue]{hyperref}
\usepackage{mathtools}
\usepackage{ifthen}
\usepackage{bm}

\usepackage{graphicx}
\usepackage{caption,rotating}
\usepackage{color}
\usepackage{subcaption}

\usepackage{enumitem}

\newcommand{\mathds}[1]{{\mathbb #1}}

\numberwithin{equation}{section}

\begin{document}
%
%
%
\theoremstyle{definition}
\newtheorem{Definition}{Definition}[section]
\newtheorem*{Definitionx}{Definition}
\newtheorem{Convention}{Definition}[section]
\newtheorem{Construction}[Definition]{Construction}
\newtheorem{Example}[Definition]{Example}
\newtheorem{Exercise}[Definition]{Exercise}
\newtheorem{Examples}[Definition]{Examples}
\newtheorem{Remark}[Definition]{Remark}
\newtheorem*{Remarkx}{Remark}
\newtheorem{Remarks}[Definition]{Remarks}
\newtheorem{Caution}[Definition]{Caution}
\newtheorem{Conjecture}[Definition]{Conjecture}
\newtheorem*{Conjecturex}{Conjecture}
\newtheorem{Question}[Definition]{Question}
\newtheorem*{Questionx}{Question}
\newtheorem*{Acknowledgements}{Acknowledgements}
\newtheorem*{Notation}{Notation}
\newtheorem*{Organization}{Organization}
\newtheorem*{Disclaimer}{Disclaimer}
\theoremstyle{plain}
\newtheorem{Theorem}[Definition]{Theorem}
\newtheorem*{Theoremx}{Theorem}

\newtheorem{Proposition}[Definition]{Proposition}
\newtheorem*{Propositionx}{Proposition}
\newtheorem{Lemma}[Definition]{Lemma}
\newtheorem{Corollary}[Definition]{Corollary}
\newtheorem*{Corollaryx}{Corollary}
\newtheorem{Fact}[Definition]{Fact}
\newtheorem{Facts}[Definition]{Facts}
\newtheoremstyle{voiditstyle}{3pt}{3pt}{\itshape}{\parindent}%
{\bfseries}{.}{ }{\thmnote{#3}}%
\theoremstyle{voiditstyle}
\newtheorem*{VoidItalic}{}
\newtheoremstyle{voidromstyle}{3pt}{3pt}{\rm}{\parindent}%
{\bfseries}{.}{ }{\thmnote{#3}}%
\theoremstyle{voidromstyle}
\newtheorem*{VoidRoman}{}

\newenvironment{specialproof}[1][\proofname]{\noindent\textit{#1.} }{\qed\medskip}
\newcommand{\blowup}{\rule[-3mm]{0mm}{0mm}}
\newcommand{\cal}{\mathcal}
\newcommand{\Aff}{{\mathds{A}}}
\newcommand{\BB}{{\mathds{B}}}
\newcommand{\CC}{{\mathds{C}}}
\newcommand{\BC}{{\mathds{C}}}
\newcommand{\EE}{{\mathds{E}}}
\newcommand{\FF}{{\mathds{F}}}
\newcommand{\GG}{{\mathds{G}}}
\newcommand{\G}{{\mathds{G}}}
\newcommand{\HH}{{\mathds{H}}}
\newcommand{\NN}{{\mathds{N}}}
\newcommand{\ZZ}{{\mathds{Z}}}
\newcommand{\BZ}{{\mathds{Z}}}
\newcommand{\PP}{{\mathds{P}}}
\newcommand{\QQ}{{\mathds{Q}}}
\newcommand{\RR}{{\mathds{R}}}
\newcommand{\BA}{{\mathds{A}}}
\newcommand{\BQ}{{\mathds{Q}}}
\newcommand{\UU}{{\mathcal U}}
\newcommand{\MM}{{\mathcal M}}
\newcommand{\xX}{{\mathcal X}}
\newcommand{\yY}{{\mathcal Y}}
\newcommand{\dD}{{\mathcal D}}
\newcommand{\kK}{{\mathcal K}}
\newcommand{\Liea}{{\mathfrak a}}
\newcommand{\Lieb}{{\mathfrak b}}
\newcommand{\Lieg}{{\mathfrak g}}
\newcommand{\Liem}{{\mathfrak m}}
\newcommand{\ideala}{{\mathfrak a}}
\newcommand{\idealb}{{\mathfrak b}}
\newcommand{\idealg}{{\mathfrak g}}
\newcommand{\idealm}{{\mathfrak m}}
\newcommand{\idealp}{{\mathfrak p}}
\newcommand{\idealq}{{\mathfrak q}}
\newcommand{\idealI}{{\cal I}}
\newcommand{\lin}{\sim}
\newcommand{\num}{\equiv}
\newcommand{\dual}{\ast}
\newcommand{\iso}{\cong}
\newcommand{\homeo}{\approx}
\newcommand{\mm}{{\mathfrak m}}
\newcommand{\pp}{{\mathfrak p}}
\newcommand{\qq}{{\mathfrak q}}
\newcommand{\rr}{{\mathfrak r}}
\newcommand{\pP}{{\mathfrak P}}
\newcommand{\qQ}{{\mathfrak Q}}
\newcommand{\rR}{{\mathfrak R}}
\newcommand{\OO}{{\mathcal O}}
\newcommand{\CO}{{\mathcal O}}
\newcommand{\numero}{{n$^{\rm o}\:$}}
\newcommand{\mf}[1]{\mathfrak{#1}}
\newcommand{\mc}[1]{\mathcal{#1}}
\newcommand{\into}{{\hookrightarrow}}
\newcommand{\onto}{{\twoheadrightarrow}}
\newcommand{\Spec}{{\rm Spec}\:}
\newcommand{\BigSpec}{{\rm\bf Spec}\:}
\newcommand{\Spf}{{\rm Spf}\:}
\newcommand{\Proj}{{\rm Proj}\:}
\newcommand{\Pic}{{\rm Pic }}
\newcommand{\MW}{{\rm MW }}
\newcommand{\Br}{{\rm Br}}
\newcommand{\NS}{{\rm NS}}
\newcommand{\Sym}{{\mathfrak S}}
\newcommand{\Aut}{{\rm Aut}}
\newcommand{\Autp}{{\rm Aut}^p}
\newcommand{\ord}{{\rm ord}}
\newcommand{\coker}{{\rm coker}\,}
\newcommand{\divisor}{{\rm div}}
\newcommand{\Def}{{\rm Def}}
\newcommand{\rank}{\mathop{\mathrm{rank}}\nolimits}
\newcommand{\Ext}{\mathop{\mathrm{Ext}}\nolimits}
\newcommand{\EXT}{\mathop{\mathscr{E}{\kern -2pt {xt}}}\nolimits}
\newcommand{\Hom}{\mathop{\mathrm{Hom}}\nolimits}
\newcommand{\Bs}{\mathop{\mathrm{Bs}}\nolimits}
\newcommand{\Bk}{\mathop{\mathrm{Bk}}\nolimits}
\newcommand{\HOM}{\mathop{\mathscr{H}{\kern -3pt {om}}}\nolimits}
\newcommand{\Exc}{\mathop{\mathrm{Exc}}\nolimits}
\newcommand{\calA}{\mathscr{A}}
\newcommand{\calC}{\mathscr{C}}
\newcommand{\calH}{\mathscr{H}}
\newcommand{\calL}{\mathscr{L}}
\newcommand{\calM}{\mathscr{M}}
\newcommand{\calN}{\mathscr{N}}
\newcommand{\calS}{\mathscr{S}}
\newcommand{\CS}{\mathcal{S}}
\newcommand{\N}{\mathcal{N}}
\newcommand{\calX}{\mathscr{X}}
\newcommand{\calK}{\mathscr{K}}
\newcommand{\calD}{\mathscr{D}}
\newcommand{\calY}{\mathscr{Y}}
\newcommand{\calF}{\mathscr{F}}
\newcommand{\f}{\mathscr{F}}
\newcommand{\e}{\mathscr{E}}
\newcommand{\calG}{\mathscr{G}}
\newcommand{\CN}{\mathcal{N}}
\newcommand{\DD}{\mathcal{D}}
\newcommand{\CCC}{\mathcal{C}}
\newcommand{\C}{\mathscr{C}}
\newcommand{\CK}{\mathcal{K}}
\newcommand{\CM}{\mathcal{M}}
\newcommand{\Jac}{\operatorname{Jac}}

\newcommand{\chari}{\mathop{\mathrm{char}}\nolimits}
\newcommand{\ch}{\mathop{\mathrm{ch}}\nolimits}
\newcommand{\CH}{\mathop{\mathrm{CH}}\nolimits}
\newcommand{\supp}{\mathop{\mathrm{supp}}\nolimits}
\newcommand{\codim}{\mathop{\mathrm{codim}}\nolimits}
\newcommand{\td}{\mathop{\mathrm{td}}\nolimits}
\newcommand{\Span}{\mathop{\mathrm{Span}}\nolimits}
\newcommand{\Gal}{\mathop{\mathrm{Gal}}\nolimits}
\newcommand{\sym}{\mathop{\mathrm{Sym}}\nolimits}
\newcommand{\cl}{\mathop{\mathrm{cl}}\nolimits}
\newcommand{\RCH}{\mathop{\mathrm{RCH}}\nolimits}
\newcommand{\Cr}{\mathop{\mathrm{Cr}}\nolimits}
\newcommand{\mult}{\mathop{\mathrm{mult}}\nolimits}
\newcommand{\piet}{{\pi_1^{\rm \acute{e}t}}}
\newcommand{\Het}[1]{{H_{\rm \acute{e}t}^{{#1}}}}
\newcommand{\Hfl}[1]{{H_{\rm fl}^{{#1}}}}
\newcommand{\Hcris}[1]{{H_{\rm cris}^{{#1}}}}
\newcommand{\HdR}[1]{{H_{\rm dR}^{{#1}}}}
\newcommand{\hdR}[1]{{h_{\rm dR}^{{#1}}}}
\newcommand{\loc}{{\rm loc}}
\newcommand{\et}{{\rm \acute{e}t}}
\newcommand{\defin}[1]{{\bf #1}}

\ifthenelse{\equal{1}{1}}{
\ifthenelse{\equal{2}{2}}{
\newcommand{\blue}[1]{{\color{blue}#1}}
\newcommand{\green}[1]{{\color{green}#1}}
\newcommand{\red}[1]{{\color{red}#1}}
\newcommand{\cyan}[1]{{\color{cyan}#1}}
\newcommand{\magenta}[1]{{\color{magenta}#1}}
\newcommand{\yellow}[1]{{\color{yellow}#1}} 
}{
\newcommand{\blue}[1]{#1}
\newcommand{\green}[1]{#1}
\newcommand{\red}[1]{#1}
\newcommand{\cyan}[1]{#1}
\newcommand{\magenta}[1]{#1}
\newcommand{\yellow}[1]{#1} 
}
}{
\newcommand{\blue}[1]{}
\newcommand{\green}[1]{}
\newcommand{\red}[1]{}
\newcommand{\cyan}[1]{}
\newcommand{\magenta}[1]{}
\newcommand{\yellow}[1]{} 
}

\newcommand{\ver}{\operatorname{vert}}

\newcommand{\eric}[1]{{\color{red} \sf $\clubsuit\clubsuit\clubsuit$ Eric: [#1]}}
\newcommand{\wern}[1]{{\color{purple} \sf $\clubsuit\clubsuit\clubsuit$ Wern: [#1]}}
\newcommand{\xichen}[1]{{\color{blue} \sf $\clubsuit\clubsuit\clubsuit$ Xi: [#1]}}

\renewcommand{\HH}{{\rm{H}}}

\title[Complements of Generic Hypersurfaces]{Algebraic Hyperbolicity of Complements of Generic Hypersurfaces in Projective Spaces}

\date{\today}

\author[X. Chen]{Xi Chen}
\address{632 Central Academic Building\\
University of Alberta\\
Edmonton, Alberta T6G 2G1, CANADA}
\email{xichen@math.ualberta.ca}

\author[E. Riedl]{Eric Riedl}
\email{eriedl@nd.edu}
\address{255 Hurley Hall \\
University of Notre Dame \\
Notre Dame, IN 46556}

\author[W. Yeong]{Wern Yeong}
\email{wyeong@nd.edu}
\address{255 Hurley Hall \\
University of Notre Dame \\
Notre Dame, IN 46556}




\begin{abstract}
We study the algebraic hyperbolicity of the complement of very general degree $2n$ hypersurfaces in $\PP^n$. We prove the Algebraic Green-Griffiths-Lang Conjecture for these complements, and in the case of the complement of a quartic plane curve, we completely characterize the exceptional locus as the union of the flex and bitangent lines.
\end{abstract}

\maketitle


\section{Introduction}

The relationship between the canonical bundle of a smooth variety $X$ and the holomorphic maps from $\mathbb{C}$ to $X$ is a fundamental driving theme in the study of these holomorphic maps. A major conjecture in this area is the Green-Griffiths-Lang Conjecture, which we state in the logarithmic form that is the focus of this paper.

\begin{Conjecture}\label{conj-GGL}
Let $X$ be a smooth projective variety and $D$ a simple normal crossings divisor such that $K_X + D$ is big. Then there is some proper subvariety $S \subsetneq X$ containing the images of all nonconstant holomorphic maps $\mathbb{C} \to X \setminus D$, i.e. $X\setminus D$ is pseudo Brody hyperbolic.
\end{Conjecture}

As it stands, a full proof of Conjecture \ref{conj-GGL} appears far out of reach. In \cite{ChenX-log}, the author introduced the following algebraic version of hyperbolicity for log varieties, which is a natural extension of Demailly's definition for projective varieties \cite{Dem97}. Given a smooth projective variety $X$ and a simple normal crossings divisor $D$, we say that the log pair $(X,D)$ is \emph{algebraically hyperbolic} if there exists an $\epsilon > 0$ such that every map $f: C\to X$ from a smooth projective curve $C$ that is birational onto its image and $f(C)\not\subset D$ satisfies $ 2g(C) - 2 + |f^{-1}(D)| > \epsilon \deg f^* L$, where $L$ is an ample line bundle on $X$ and $g(C)$ is the geometric genus of $C$. 
The following weaker version of Conjecture \ref{conj-GGL} with Brody hyperbolicity replaced by algebraic hyperbolicity has turned out to be more tractable. 

\begin{Conjecture}\label{conj-algGGL}
Let $X$ be a smooth projective variety and $D$ be a simple normal crossings divisor such that $K_X + D$ is big. Then there exists an $\epsilon > 0$ and a subvariety $S \subsetneq X$ such that every map $f: C\to X$ from a smooth projective curve $C$ that is birational onto its image and $f(C)\not\subset D\cup S$ satisfies
$$ 2g(C) - 2 + |f^{-1}(D)| > \epsilon \deg f^* L$$
where $L$ is an ample line bundle on $X$ and $g(C)$ is the geometric genus of $C$.
\end{Conjecture}

This conjecture is closely related to some conjectures of Vojta \cite{vojta1999integral}, and is sometimes called the algebraic Lang-Vojta Conjecture for function fields \cite{CorvajaZannier2008}, although neither Lang nor Vojta put their conjectures in this form. See \cite{JavanpeykarSurvey} for a survey on various notions of hyperbolicity. We prove Conjecture \ref{conj-algGGL} for very general hypersurfaces in $\PP^n$ of degree $d = 2n$ and give a complete description of the exceptional locus $S$ in the case $n=2$, $d=4$.

\begin{Theorem}\label{thm-varmain}
    For a very general hypersurface $D\subseteq \PP^n$ of degree $d=2n$, there is some proper exceptional locus $S\subsetneq \PP^n$ such that for every map $f: C\to \PP^n$ from a smooth projective curve $C$ that is birational onto its image and $f(C)\not\subset D\cup S$, we have: \begin{itemize}
        \item for $n \geq 3$ \begin{equation}\label{eqn-logah3}
        2g(C)-2+|f^{-1}(D)|\geq \deg f^* L
    \end{equation}
    \item for $n=2$, \begin{equation}\label{eqn-logah2}
        2g(C)-2+|f^{-1}(D)| \geq \frac{1}{2} \deg f^*L
    \end{equation}
    \end{itemize}
    where $L = \OO_{\PP^n}(1)$ is the hyperplane bundle of $\PP^n$.
    In the case $n=2$, $S\subsetneq \PP^2$ is exactly the union of the bitangent and flex lines to $D$.
\end{Theorem}

Zaidenberg \cite{Zaidenberg} shows that for $d \leq 2n$, the set $S$ cannot be empty, as there exist lines meeting $D$ set-theoretically in at most two points.

Theorem \ref{thm-varmain} proves a well-known conjecture of Lang about semi-abelian varieties in this case (Conjecture 2.7 in \cite{lang1986hyperbolic}).

\begin{Corollary}
\label{cor-LangCor}
Any nonconstant map from a semiabelian variety to the complement of a very general quartic in $\PP^2$ must have image in the locus of flex and bitangent lines.
\end{Corollary}

In fact, the following stronger statement follows from Corollary \ref{cor-LangCor} combined with Lemmas 3.15 and 3.16 of \cite{JavanpeykarXie}.

\begin{Corollary}
Any nonconstant map from a connected algebraic group over $\CC$ to the complement of a very general quartic $D$ in $\PP^2$ must have image in the locus $\Delta$ of flex and bitangent lines, i.e. $\PP^2\setminus D$ is groupless modulo $\Delta$.    
\end{Corollary}

Conjectures \ref{conj-GGL} and \ref{conj-algGGL} are still open in dimension two and above, even for complements of smooth hypersurfaces in projective space. See \cite{RousTurWangForum, RousTurWang2021} for some further discussion about related conjectures on the arithmetic of complements of divisors. There has been a lot of work done on proving special cases of these conjectures, both for very general hypersurfaces in projective space and for the complements of plane curves with sufficiently many components.

Work on the algebraic hyperbolicity of very general hypersurfaces in projective space started with Clemens' seminal work \cite{C}. His techniques have been improved in \cite{E1}, \cite{E2}, \cite{V1}, \cite{V2}, \cite{Pac03}, \cite{CR04}, \cite{RY16} and many more and adapted to the log case
in \cite{PR-log}. For a very general hypersurface $D\subset \PP^n$ of degree $d$,
$$
2g(C) - 2 + i_X(C, D) \ge (d-2n) \deg C
$$
for every integral curve $C\subset \PP^n$ and $C\not\subset D$, where $g(C)$ is the geometric genus of $C$ and $i_X(C,D)$ is the number of points in $\nu^{-1}(D)$ under the normalization $\nu: C^\nu\to X$ of $C$ \cite{ChenX-log, PR-log}. Using this result, Roulleau and Rousseau prove algebraic hyperbolicity of certain cyclic covers of projective space \cite{RoulleauRousseau}. There is also important work done on the Brody hyperbolicity of hypersurfaces and the Kobayashi Conjecture, including \cite{DivMerkRouss}, \cite{Brotbek}, \cite{BrotbekDeng}, \cite{BercziKirwan}, \cite{RiedlYangAGTH}. 

The (algebraic) hyperbolicities of $(X, D)$ and a cyclic cover $\widehat{X}$ of $X$ ramified over $D$ are related. It is obvious that the (pseudo) algebraic hyperbolicity of $\widehat{X}$ implies that of $(X,D)$. To derive the converse, we need to show that
$$ 2g(C) - 2 + |f^{-1}(D)| \ge (\varepsilon + \dfrac{1}{m}) \deg f^* D $$
for some $\varepsilon > 0$, assuming that $D$ is ample, where $m$ is the degree of $\widehat{X}\to X$. So our result just fails to imply the pseudo algebraic hyperbolicity of a degree $d$ cyclic over of $\PP^n$ ramified over a very general hypersurface of degree $d=2n$.

For the complement of special plane curves, much of the work focuses on curves which consist of several irreducible components. After pioneering work by Green \cite{Green}, further results include \cite{DethloffSchumacherWong}, \cite{Babets}, \cite{CorvajaZannier2008}, \cite{CorvajaZannier}.

In Section \ref{sec-variational}, we first employ a variant of an argument used by many others (including Pacienza and Rousseau in \cite{PR-log}) to characterize the possible curves violating algebraic hyperbolicity. We then make a few key observations about the stability of the Lazarsfeld bundle $M_1$ that allow us to characterize the possible exceptional curves as coming from the fibers of a rational map $\theta$. In Section \ref{sec-fixedModulus}, we then study the specific geometry of these fibers for the case $n=2$, showing that the only curves violating (\ref{eqn-logah2}) are the bitangent and flex lines.

\subsection*{Acknowledgements}
We gratefully acknowledge helpful conversations with Kenny Ascher, Izzet Coskun, Ariyan Javanpeykar, Emanuela Marangone, and Amos Turchet. We are also grateful for many helpful suggestions by the anonymous referee. Eric Riedl is supported by NSF CAREER grant DMS-1945944. Xi Chen is supported by Discovery Grant RGPIN-2019-04775 from the Natural Sciences and Engineering Research Council of Canada.

\section{Notation and conventions}

\subsection{Lines with fixed modulus}
A key consideration in this paper is the modulus of intersection of a line with a given hypersurface. Recall that $\overline{M}_{0,d}$ is the moduli space of stable rational curves with $d$ marked points. It is a projective variety, and forms a fine moduli space for families of stable rational curves. After taking a quotient by the natural $S_d$ action, we obtain the space of stable rational curves with $d$ unordered marked points, which we call $\overline{M}_{0,d}^{\sym}$. Similarly, we have the space $\overline{M}_{0,d+1}$ and can take its quotient by the $S_d$ action on the first $d$ points and get a space $\overline{U}_{0,d}^{\sym}$. The space $\overline{M}_{0,d}^{\sym}$ is a coarse moduli space for families of stable rational curves with $d$ marked, unordered points, and $\overline{U}_{0,d}^{\sym}$ is the coarse moduli space for families of stable rational curves with one special marked point together with $d$ other, unordered marked points.

Let $V \subset \PP^n \times \mathbb{G}(1,n)$ be the universal line in $\PP^n$, and let $\pi$ be the projection to $\PP^n$ and $\eta$ the projection to $\mathbb{G}(1,n)$. Given a hypersurface $D \subset \PP^n$, we have a rational map $\theta: V \dashrightarrow \overline{U}_{0,d}^{\sym}$ given by sending a pair $(p,\ell)$ to the element $(D \cap \ell, p)$, where the $d$ unordered points are given by $D \cap \ell$ and the $(d+1)$st point is given by $p$.

This map will be particularly important in the case $d=4$, $n=2$, where we study the fibers of $\theta$ in detail. In this case the map $\theta$ is defined everywhere except for pairs $(p,\ell)$ where $p \cup \ell \cap D$ has a multiplicity three point along the line $\ell$, where it isn't clear which stable curve corresponds to $\ell \cap D$. In particular, this means that $\theta$ is undefined along the flex locus, or along tangent lines to $D$ when the point $p$ coincides with the point of tangency.

\subsection{Log tangent sheaves and Lazarsfeld bundles}
We briefly describe some important sheaves that we use in this paper, and highlight a few of their important properties. We begin by reviewing \emph{log tangent sheaves}. Let $X$ be a smooth projective variety and $D$ a divisor in $X$. Then the log tangent sheaf of $D$ in $X$ is defined by the exact sequence
\[ \begin{tikzcd} 0 \ar{r} & T_X(-\log D) \ar{r} & T_X \ar{r} & \OO_D(D) \ar{r} & 0. \end{tikzcd}\]
If $D$ is a simple normal crossing divisor (so in particular, if $D$ is smooth), then $T_X(-\log D)$ is a vector bundle. Given a smooth subvariety $Y$ of $X$, we also have the \emph{log normal sheaf} $N_{Y/X}(\log D)$, defined by the following short exact sequence:
\[ \begin{tikzcd} 0 \ar{r} & T_Y(-\log D) \ar{r} & T_X(-\log D)|_Y \ar{r} & N_{Y/X}(\log D) \ar{r} & 0. \end{tikzcd} \]
Recall that if $f: C \to X$ is a map from a smooth curve to $X$, the space of deformations of $f \in \Hom(C,X)$ that preserve $f^{-1}(D)$ is given by $H^0(f^*T_X(-\log D))$ (Proposition 5.3 in \cite{KeelMcKernan}). Similarly, the space of deformations of the pair $(f,C)$ in Kontsevich space that preserve the combinatorial type of the divisor $f^{-1}(D)$ are given by $H^0(N_{f/X}(\log D))$ (see \cite{Olsson-logcot} Theorem 5.6).

We also use Lazarsfeld bundles, also known as syzygy bundles \cite{ELM}. We denote by $M_d$ the \emph{Lazarsfeld bundle associated to $\OO_{\PP^n}(d)$}, which is defined as the kernel in the short exact sequence:
\begin{equation}\label{eqn-lm}
\begin{tikzcd}
0 \ar{r} & M_d \ar{r} & H^0(\PP^n,\OO_{\PP^n}(d)) \otimes \OO_{\PP^n} \ar{r}{\text{ev}} & \OO_{\PP^n}(d) \ar{r} & 0.
\end{tikzcd}
\end{equation}
These are vector bundles on $\PP^n$, and the fiber of $M_d$ over a point $p$ of $\PP^n$ is the space of degree $d$ polynomials on $\PP^n$ that vanish at $p$.

\subsection{Dominant families of curves}
Our last result shows that dominant families of curves satisfy a strong form of algebraic hyperbolicity. Thus, the challenge in proving Conjecture \ref{conj-algGGL} lies in showing that all curves violating the inequality are contained in some subvariety, not some countable union of subvarieties.

\begin{Proposition}\label{prop-algGGLfordominantFamiliesOfCurves}
    Let $X$ be a smooth variety and $D$ be a smooth divisor in $X$ and $f: C \to X$ be the general element of a family of curves with fixed genus $g$ and number of intersection points $i$ with $D$, such that the family of curves dominates $X$. Then 
    \[ 2g-2 + i \geq (K_X+D) \cdot f(C) .\]
\end{Proposition}
\begin{proof}
    Because $C$ is a general element of a dominant families of curves in $X$, the log normal sheaf $N_{f/X}(\log D)$ is generically globally generated, and hence has non-negative degree. The degree of the log normal sheaf is $2g-2+i - (K_X+D) \cdot f(C)$. The result follows.
\end{proof}

\section{A variational argument}
\label{sec-variational}

In this section we study the curves that may violate algebraic hyperbolicity. We show that to every point in such a curve, there is an associated line $\ell$, and that such lines either intersect $D$ with fixed modulus or intersect $D$ in few points. 
More precisely, we show the following.

\begin{Theorem} \label{thm-associatedLine}
    Let $D \subset \PP^n$ be a very general hypersurface of degree $d=2n$. Then for any curve $C$ in $\PP^n$ not contained in $D$, we have either \begin{enumerate}
        \item $C$ satisfies one of the following conditions \begin{eqnarray*}
        2g(C)-2+i(C,D)\geq \deg C \text{ and $n \geq 3$}  \\
        2g(C)-2+i(C,D)\geq \frac{1}{2} \deg C \text{ and $n = 2$}
    \end{eqnarray*}
    or
    \item $C$ lifts to $V$ and is contained in either a fiber or the undefined locus of the rational map $\theta: V \dashrightarrow \overline{U}_{0,d}^{\sym}$.
    \end{enumerate}
\end{Theorem}

    

\subsection{Setup}
Our first goal is to set up the basic construction used by Voisin, Pacienza, Rousseau, Clemens, and Ran, reinterpreted as in \cite{CoskunRiedl, Yeong}. The main result of this subsection is Proposition \ref{prop-mainSetupProp}.

Suppose that for a very general section $D$ of $\OO_{\PP^n}(d)$, there is an integral curve $C$ in $\PP^n$ of geometric genus $g$ and degree $e$ such that $C\nsubseteq D$ and $i(C,D)=i$. Let $B_1=H^0(\PP^n, \OO_{\PP^n}(d))$. By passing to an \'etale cover $B_2$ of $B_1$, we can construct the following families:
\begin{enumerate}
	\item $\xX=\PP^n\times B_2$,
	\item the universal degree-$d$ hypersurface $\dD\subseteq \xX$ over $B_2$, and
	\item a family $\yY\subseteq \xX$ over $B_2$ of pointed curves of geometric genus $g$ and degree $e$, 
\end{enumerate}
such that $\yY\nsubseteq \dD$ and at a generic fiber parametrized by $b\in B_2$, we have $i(Y_b,D_b)=i$. After taking a resolution of $\yY/B_2$ and possibly restricting to an open set $B \subset B_2$, we obtain a smooth family $\yY/B$ and a generically injective morphism $f:\yY\rightarrow \xX$ over $B$ such that $f^{-1}(\dD)$ is a disjoint union of sections of $\yY/B$. Let us denote by $g:\dD\rightarrow \xX$ and $h:f^{-1}(\dD)\rightarrow \yY$ the natural injective morphisms over $B$. We denote the general fibers of $\yY, \dD$ and $\xX$ over $b\in B$ by $Y_b, D_b$ and $X_b=\PP^n$, respectively, and the restriction of $f$ by $f_b:Y_b\rightarrow X_b=\PP^n$. We denote by $p_1:\xX\rightarrow \PP^n$ and $p_2:\xX\rightarrow B$ the natural projection maps from $\xX$.

Our interest in log tangent sheaves is explained by the following result.

\begin{Proposition}
Let $X$ be a smooth projective variety and $D$ a smooth divisor in $X$. For $Y$ a smooth curve mapping via $f$ to $X$, we have that the degree of $N_{Y/X}(\log D)$ is equal to $2g(Y)-2+|f^{-1}(D)|-(K_X+D) \cdot Y$. In particular, for $X= \PP^n$, $D = D_b$, $i = |f^{-1}(D)|$ and $Y=Y_b$ as above, we have
\begin{equation}\label{eqn-degree}
\begin{aligned}
\deg N_{Y_b/\PP^n}(\log D_b) &=2g-2+i-(K_X + D)f_* Y_b\\
&=2g-2+i-(d-(n+1))e.
\end{aligned}
\end{equation}
\end{Proposition}
\begin{proof}
We have the following sequences:
\begin{equation}\label{eqn-logtangent1}
\begin{tikzcd}
0 \ar{r} & T_{X}(-\log D)\ar{r} & T_{X} \ar{r} & \OO_{D}(D)\ar{r} & 0
\end{tikzcd}
\end{equation}
\begin{equation}\label{eqn-logtangent2}
\begin{tikzcd}
0 \ar{r} & T_{Y}(-\log D)\ar{r} & T_{Y} \ar{r} & \OO_{f^{-1}(D)}(f^{-1}(D))\ar{r} & 0
\end{tikzcd}
\end{equation}
\begin{equation}\label{eqn-lognormal}
\begin{tikzcd}[column sep=small]
0 \ar{r} & T_{Y}(-\log D)\ar{r} & f^*T_{X}(-\log D) \ar{r} & N_{Y/X}(\log D)\ar{r} & 0
\end{tikzcd}
\end{equation}

Taking degrees in (\ref{eqn-lognormal}), we get
\begin{equation} 
\deg N_{Y/X}(\log D)=2g-2+i-(K_X + D)f_* Y.
\end{equation}
\end{proof}

Thus, we can bound $2g-2+i$ if we can understand the degree of $N_{Y_b/\PP^n}(\log D_b)$, so it makes sense to study these log normal sheaves and log tangent sheaves. We will work with the log tangent sheaves $T_\xX(-\log \dD)$ and $T_\yY(-\log \dD)$, given by the following short exact sequences:
\begin{equation}\label{eqn-logtangent1fam}
\begin{tikzcd}
0 \ar{r} & T_\xX(-\log \dD)\ar{r} & T_\xX \ar{r} & \OO_\dD(\dD)\ar{r} & 0\\
0 \ar{r} & T_\yY(-\log \dD)\ar{r} & T_\yY \ar{r} & \OO_{f^{-1}(\dD)}(f^{-1}(\dD))\ar{r} & 0
\end{tikzcd}
\end{equation}
After possibly replacing $B$ by an open set, we may assume that $(\xX,\dD)$ and $(\yY,f^{-1}(\dD))$ are smooth pairs, and so the log tangent sheaves $T_\xX(-\log \dD)$ and $T_\yY(-\log \dD)$ are in fact vector bundles on $\xX$ and $\yY$, respectively.

We will also use the log normal sheaf $N_{\yY/\xX}(\log \dD)$, which appears in the following short exact sequence:
\begin{equation}\label{eqn-lognormalfam}
\begin{tikzcd}[column sep=small]
0 \ar{r} & T_\yY(-\log \dD)\ar{r} & f^*T_\xX(-\log \dD) \ar{r} & N_{\yY/\xX}(\log \dD)\ar{r} & 0
\end{tikzcd}
\end{equation}

The above construction can be done so that the families $\dD$ and $\yY$ are stable under the action of $GL(n+1)$. Hence, we can define the \emph{vertical log tangent sheaves} as the kernels $T_\xX^{\ver}(-\log \dD)$ and $T_\yY^{\ver}(-\log \dD)$ in the following short exact sequences:
\begin{equation}\label{eqn-verticalx}
\begin{tikzcd}
0 \ar{r} & T_\xX^{\ver}(-\log \dD)\ar{r} & T_\xX(-\log \dD) \ar{r} & p_1^* T_{\PP^n}\ar{r} & 0
\end{tikzcd}
\end{equation} 
\begin{equation}\label{eqn-verticaly}
\begin{tikzcd}
0 \ar{r} & T_\yY^{\ver}(-\log \dD)\ar{r} & T_\yY(-\log \dD) \ar{r} & f^*p_1^*T_{\PP^n}\ar{r} & 0
\end{tikzcd}
\end{equation}

The main goal of this subsection is the following Proposition.

\begin{Proposition}\label{prop-mainSetupProp}
Suppose that for a very general hypersurface $D$ of degree $d$ in $\PP^n$ there exists a smooth curve $Y$ of genus $g$ mapping birationally to $\PP^n$ with image of degree $e$ and meeting $D$ precisely $i$ times. Then there exists an \'{e}tale map $B \to H^0(\OO_{\PP^n}(d))$ and a $GL(n+1)$-invariant family $\yY$ over $B$ such that a general fiber $Y_b$ of $\mathcal{Y} \to B$ is a smooth genus $g$ curve mapping birationally to $\PP^n$ with degree $e$ and meeting $D_b$ precisely $i$ times. Furthermore, for a general $b \in B$ there is an exact sequence of sheaves
\[
\begin{tikzcd}[column sep=small]
0 \ar{r} & T_{\yY}^{\ver}(-\log \dD)|_{Y_b} \ar{r} & f^*T_{\xX}^{\ver}(- \log \dD)|_{Y_b} \ar{r} & N_{Y_b/\PP^n}(\log D_b) \ar{r} & 0.
\end{tikzcd} 
\]
\end{Proposition}

To prove it, we begin proving relations among various log tangent sheaves.

\begin{Proposition}\label{prop-2.2}
	The quotient $\left(f^*T_\xX^{\ver}(-\log \dD)\right)\big/\left(T_\yY^{\ver}(-\log \dD)\right)$ of vertical log tangent sheaves is isomorphic to $N_{\yY/\xX}(\log \dD)$.
\end{Proposition}

\begin{proof}
	The short exact sequences  (\ref{eqn-lognormalfam}), (\ref{eqn-verticalx}) and (\ref{eqn-verticaly}) fit together in the following commutative diagram, where $\kK_1$ denotes the cokernel of the natural inclusion $T_\yY^{\ver}(-\log \dD) \hookrightarrow f^*T_\xX^{\ver}(-\log \dD)$:
	\begin{equation*}
		\begin{tikzcd}[column sep=16pt]
		& &0 &0 \\
		& &f^*p_1^*T_{\PP^n} \arrow[u] \arrow[r,equal] &f^*p_1^*T_{\PP^n} \arrow[u] \\
		&0 \arrow[r] &T_\yY(-\log \dD) \arrow[u] \arrow[r] &f^*T_\xX(-\log \dD) \arrow[u] \arrow[r] &N_{\yY/\xX}(\log \dD) \arrow[r] &0\\
		&0 \arrow[r] &T_\yY^{\ver}(-\log \dD) \arrow[u] \arrow[r] &f^*T_\xX^{\ver}(-\log \dD) \arrow[u] \arrow[r] &\kK_1 \arrow[r] \arrow[u, "\simeq"] &0\\		
		& &0 \arrow[u] &0 \arrow[u] & &
		\end{tikzcd}
	\end{equation*}
	The required isomorphism follows from the Nine Lemma.	
\end{proof}

\begin{Proposition}\label{prop-pullback}
	For a generic $b\in B$, we have the isomorphism 
	$$\iota_b^*N_{\yY/\xX}(\log \dD)\simeq N_{Y_b/\PP^n}(\log D_b),$$ 
	where $\iota_b:Y_b\rightarrow \yY$ denotes the inclusion of the fiber.
\end{Proposition}

\begin{proof}
	Recall that $T_\yY(-\log \dD)$ and $T_\xX(-\log \dD)$ are vector bundles. By generic flatness, the restriction of (\ref{eqn-lognormalfam}) to the general fiber $Y_b$ remains exact. Hence, we have the short exact sequence
	\begin{equation*}
		0 \rightarrow \iota_b^*T_\yY(-\log \dD)\rightarrow \iota_b^*f^*T_\xX(-\log \dD) \rightarrow \iota_b^*N_{\yY/\xX}(\log \dD)\rightarrow 0,
	\end{equation*}
	which together with (\ref{eqn-lognormal}) extend to the exact diagram:
	\begin{equation*}
		\begin{tikzcd}[column sep=18pt]
		& &0 &0 \\
		& &\OO_{Y_b}^N \arrow[u] \arrow[r,equal] &\OO_{Y_b}^N \arrow[u] \\
		&0 \arrow[r] &\iota_b^*T_\yY(-\log \dD) \arrow[u] \arrow[r] &\iota_b^*f^*T_\xX(-\log \dD) \arrow[u] \arrow[r] &\iota_b^*N_{\yY/\xX}(\log \dD) \arrow[r] &0\\
		&0 \arrow[r] &T_{Y_b}(-\log D_b) \arrow[u] \arrow[r] &f_b^*T_{\PP^n}(-\log D_b) \arrow[u] \arrow[r] &N_{Y_b/\PP^n}(\log D_b) \arrow[r] \arrow[u, "\simeq"] &0\\		
		& &0 \arrow[u] &0 \arrow[u] & &
		\end{tikzcd}
	\end{equation*}
	Note that $N=\dim B$, and that the columns come from $Y_b$ and $D_b$ being fibers of the fibrations $\yY$ and $\dD$, respectively. Therefore, the required isomorphism follows from applying the Nine Lemma to this diagram.
\end{proof}

\begin{proof}[Proof of Proposition \ref{prop-mainSetupProp}]
We already constructed $B$ and $\yY$ in the discussion at the beginning of this section. It remains to show that the log tangent sheaf $N_{Y_b/\PP^n}(\log D_b)$ is the cokernel of the map $T_{\yY}^{\ver}|_{Y_b} \to f^* T_{\xX}^{\ver}(-\log \dD)|_{Y_b} $. By Proposition \ref{prop-pullback}, $N_{Y_b/\PP^n}(\log D_b)$ is the restriction of $N_{\yY/\xX}(\log \dD)$ to $Y_b$. By Proposition \ref{prop-2.2}, we see that $N_{\yY/\xX}(\log \dD)$ is the cokernel of the map of vertical log tangent sheaves. The result follows.
\end{proof}

\subsection{Lazarsfeld bundles}

We now work toward interpreting the sequence of Proposition \ref{prop-mainSetupProp} in terms of polynomials. The main result of this subsection is the following.

\begin{Proposition}\label{prop-mainLazarsfeldProp}
Suppose that for a very general hypersurface $D_b$ of degree $d$ in $\PP^n$ there exists a smooth curve $Y_b$ of genus $g$ mapping birationally to $\PP^n$ with image of degree $e$ and $|f^{-1}(D_b)| = i$. Then as in Proposition \ref{prop-mainSetupProp}, there is an exact sequence of sheaves
\[
\begin{tikzcd}[column sep=small]
0 \ar{r} & T_{\yY}^{\ver}(-\log \dD)|_{Y_b} \ar{r} & f^*T_{\xX}^{\ver}(- \log \dD)|_{Y_b} \ar{r} & N_{Y_b/\PP^n}(\log D_b) \ar{r} & 0.
\end{tikzcd} 
\]
If $2g-2+i < (d-(2n-1))e$, then $f_b^* p_1^* M_d$ dominates $N_{Y_b/\PP^n}(\log D_b)$ under the natural inclusion $M_d \to T_{\xX}(- \log \dD)$. Thus, for some general polynomials $P_1, \dots, P_s$, we have a surjection $f_b^*p_1^* M_1^{\oplus s} \to N_{Y_b/\PP^n}(\log D_b)$.
\end{Proposition}

Our first step is to compare the log vertical tangent sheaf of $\xX$ to the Lazarsfeld bundle $M_d$.

\begin{Proposition}\label{prop-LMinclusion}
	The pullback $p_1^*M_d$ of the Lazarsfeld bundle naturally injects into the vertical log tangent sheaf $T_\xX^{\ver}(-\log \dD)$, and the cokernel of $p_1^*M_d \hookrightarrow T_\xX^{\ver}(-\log \dD)$ is $p_1^*\OO_{\PP^n}.$
\end{Proposition}

\begin{proof}
First, we have the following commutative diagram, which contains the short exact sequences (\ref{eqn-logtangent1fam}) and (\ref{eqn-verticalx}). $\kK_2$ denotes the cokernel of the natural inclusion $T_\xX^{\ver}(-\log \dD) \hookrightarrow p_2^*T_B\simeq H^0(\OO_{\PP^n}(d))\otimes \OO_\xX$. The middle column comes from the natural splitting of $T_\xX\simeq p_1^*T_{\PP^n}\oplus p_2^*T_B$.
\begin{equation}\label{diag-1}
	\begin{tikzcd}[column sep=14pt, row sep=12pt]
		& &0 &0 \\
		& &p_1^*T_{\PP^n} \arrow[u] \arrow[r,equal] &p_1^*T_{\PP^n} \arrow[u] \\
		&0 \arrow[r] &T_\xX(-\log \dD) \arrow[u] \arrow[r] &T_\xX \arrow[u] \arrow[r] &\OO_\dD(\dD) \arrow[r] &0\\
		&0 \arrow[r] &T_\xX^{\ver}(-\log \dD) \arrow[u] \arrow[r] &p_2^*T_B\simeq H^0(\OO_{\PP^n}(d))\otimes \OO_\xX \arrow[u] \arrow[r] &\kK_2 \arrow[r] \arrow[u, "\simeq"] &0\\		
		& &0 \arrow[u] &0 \arrow[u] & &
	\end{tikzcd}
\end{equation}
It follows from the Nine Lemma that $\kK_2\simeq \OO_\dD(\dD)$.

The bottom row of the above diagram induces a natural inclusion $p_1^*M_d \hookrightarrow T_\xX^{\ver}(-\log \dD)$ via the natural restriction map $p_1^*\OO_{\PP^n}(d) \to \OO_{\mathcal D}(\mathcal D)$. This gives the following commutative diagram, which contains the bottom row of (\ref{diag-1}) and the pullback of the short exact sequence (\ref{eqn-lm}) via the projection $p:\xX\rightarrow \PP^n$. $\kK_3$ denotes the cokernel of $p_1^*M_d \hookrightarrow T_\xX^{\ver}(-\log \dD)$. The right column is obtained by twisting the standard short exact sequence $0 \rightarrow \OO_\xX(-\dD) \rightarrow \OO_\xX \rightarrow \OO_\dD\rightarrow 0$
by $\OO_\xX(\dD)$. Observe that although the first two terms in the right column are both sheaves pulled back from $\PP^n$, the map between them is not pulled back from $\PP^n$.
\begin{equation}\label{diag-2}
	\begin{tikzcd}[column sep=13pt, row sep=12pt]
		& & & &0 \arrow[d]\\
		& &0 \arrow[d] & &p_1^*\OO_{\PP^n} \arrow[d]\\
		&0 \arrow[r] &p_1^*M_d \arrow[d] \arrow[r] &H^0(\OO_{\PP^n}(d)) \otimes \OO_\xX\arrow[r] \arrow[d,"="] &p_1^*\OO_{\PP^n}(d) \arrow[r] \arrow[d] &0\\
		&0 \arrow[r] &T_\xX^{\ver}(-\log \dD) \arrow[d] \arrow[r] &H^0(\OO_{\PP^n}(d))\otimes \OO_\xX \arrow[r] &\OO_{\dD}(\dD) \arrow[r] \arrow[d] &0\\
		& &\kK_3 \arrow[d] & &0\\
		& &0
	\end{tikzcd}
\end{equation}

It follows from the Snake Lemma that $\kK_3\simeq p_1^*\OO_{\PP^n}$.

\end{proof}


By Proposition \ref{prop-2.2}, there is a surjective map $$f^*T_\xX^{\ver}(-\log \dD) \rightarrow N_{\yY/\xX}(\log \dD).$$
Let $\alpha:f^*p_1^*M_d\rightarrow N_{\yY/\xX}(\log \dD)$ denote the restriction of this map to $f^*p_1^*M_d$. Viewed together with the pullback of the left column of (\ref{diag-2}) via $f:\yY\rightarrow \xX$, we obtain the following commutative diagram:

\begin{equation}\label{diag-3}
	\begin{tikzcd}[column sep=18pt]
	&0 \arrow[r] &f^*p_1^*M_d \arrow[d, "\alpha"] \arrow[r] &f^*T_\xX^{\ver}(-\log \dD) \arrow[d] \arrow[r] &f^*p_1^*\OO_{\PP^n} \arrow[d] \arrow[r] &0\\
	&0 \arrow[r] &\text{Im} \;\alpha \arrow[d] \arrow[r] &N_{\yY/\xX}(\log \dD) \arrow[r] \arrow[d] &\coker \alpha \arrow[r] \arrow[d] &0\\
	& &0 &0 &0
	\end{tikzcd}
\end{equation}

For $b$ general, the diagram (\ref{diag-3}) pulls back via $\iota_b:Y_b\hookrightarrow \yY$ to the following exact diagram on $Y_b$:

\begin{equation}\label{diag-4}
	\begin{tikzcd}[column sep=18pt]
		&0 \arrow[r] &f_b^*p_1^*M_d \arrow[d, "\alpha_b"] \arrow[r] &f_b^*T_\xX^{\ver}(-\log \dD) \arrow[d] \arrow[r] &f_b^*p_1^*\OO_{\PP^n} \arrow[d] \arrow[r] &0\\
		&0 \arrow[r] &\iota_b^*\text{Im} \;\alpha \arrow[d] \arrow[r] &N_{Y_b/\PP^n}(\log D_b) \arrow[r] \arrow[d] &\iota_b^*\coker \alpha \arrow[r] \arrow[d] &0\\
		& &0 &0 &0
	\end{tikzcd}
\end{equation}

It follows from the surjection in the right column that 
\begin{equation}\label{eqn-cokerdegree}
	\deg \iota_b^*\coker \alpha \geq 0.
\end{equation}
So in order to obtain a lower bound for $\deg N_{Y_b/\PP^n}(\log D_b)$ in light of (\ref{eqn-degree}), we focus on the sheaf $\iota_b^*\text{Im} \;\alpha$ instead. The following two propositions express the idea that for this purpose, we can view $M_1$ as a simpler building block for $\iota_b^*\text{Im} \;\alpha$.

\begin{Proposition}\label{prop-secdom}
	For some integer $s>0$, there is a surjection 
	\begin{equation*}
	M_1^{\oplus s} \longrightarrow M_d
	\end{equation*} 
	given by multiplication by generic polynomials of degree $d-1.$
\end{Proposition}
\begin{proof}
Let $P_1, \dots, P_s$ be a basis for $H^0(\OO_{\PP^n}(d-1))$. Then multiplication by $P_i$ gives a map $M_1 \to M_d$. Taking the direct sum of all such maps, we get a surjection $M_1^{\oplus s} \to M_d$.
\end{proof}

\begin{Proposition}\label{prop-bounds}
	Let $b\in B$ be a generic element, and let $N$ be a sheaf on the curve $Y_b$. Suppose that there is a generically surjective map
	\begin{equation*}
	\beta:f_b^*p_1^*M_1^{\oplus t}\longrightarrow N
	\end{equation*} for some $t$. Then, 
	\begin{equation}\label{eqn-bounds1}
	\deg N\geq -te.
	\end{equation}
	In particular, setting $N=\iota_b^*\text{Im} \;\alpha,$ we obtain
	\begin{equation}\label{eqn-bounds2}
	2g-2+i\geq (d-(n+1+t))e.
	\end{equation}
\end{Proposition}

\begin{proof}
	Since $\beta$ is generically surjective and $M_1$ injects into a direct sum of trivial sheaves, we have
	\begin{equation*}
	\deg N \geq \deg \beta\left(f_b^*p_1^*M_1^{\oplus t}\right)
	\geq t\cdot \deg f_b^*p_1^*M_1=-te.
	\end{equation*} If $N=\iota_b^*\text{Im} \;\alpha$, then by (\ref{eqn-degree}) and (\ref{eqn-cokerdegree}), we obtain
	\begin{align*}
	2g-2 +i &= (K_X + D_b)\cdot f_* Y_b + \deg N_{Y_b/\PP^n}(\log D_b)\\
	&\geq (d-(n+1))e - te.
	\end{align*}
\end{proof}

If the image of the curve $Y_b$ in $\PP^n$ has cusps, the sheaf $N_{Y_b/\PP^n}(\log D_b)$ will have torsion. However, we can still estimate a lower bound for the degree of the sheaf. Using Proposition \ref{prop-bounds}, we now obtain estimates for $t$ depending on whether or not $\alpha:f^*p_1^*M_d\rightarrow N_{\yY/\xX}(\log\dD)$ is generically surjective.

\begin{Proposition}\label{prop-notsurj}
    Suppose $\alpha$ is not generically surjective. Then for generic $b\in B$, the curve $f_b:Y_b\rightarrow X_b=\PP^n$ satisfies 
    $$2g-2+i\geq (d-(2n-1)) e.$$
    In particular, when $d=2n,$ the curve satisfies 
    $$2g-2+i\geq e.$$
\end{Proposition}

\begin{proof}
    By Proposition \ref{prop-secdom}, there is a surjection $M_1^{\oplus s}\rightarrow M_d$ for some $s$. We can pull back this surjection via $f_b^*p_1^*$, and then compose with the surjection $\alpha_b$ from diagram (\ref{diag-4}) to get a generically surjective map
    \begin{equation*}
    	f_b^*p_1^*M_1^{\oplus t}\rightarrow \iota_b^*\text{Im }\alpha
    \end{equation*}
    for some $t\leq s$. Since $\alpha$ is not generically surjective, we must have \linebreak rank $\iota_b^*\text{Im }\alpha<$ rank $N_{Y_b/\PP^n}(\log D_b) = n-1$, so we can take $t\leq n-2$. Proposition \ref{prop-bounds} then implies that $\deg \iota_b^*\text{Im }\alpha\geq -(n-2)e.$ Hence by (\ref{eqn-cokerdegree}), we have $\deg N_{Y_b/\PP^n}(\log D_b)\geq-(n-2)e$, and so $$2g-2+i\geq (d-(2n-1))e.$$
\end{proof}

\begin{proof}[Proof of Proposition \ref{prop-mainLazarsfeldProp}]
The first part follows from Proposition \ref{prop-mainSetupProp}, and the natural inclusion $M_d \to T_{\mathcal X}(-\log \mathcal{D})$ comes from Proposition \ref{prop-LMinclusion}. By Proposition \ref{prop-notsurj}, this map will be surjective if $2g-2+i < (d-(2n-1))e$. The last statement follows directly from Proposition \ref{prop-secdom}.
\end{proof}

\subsection{Stability and associated lines}

Now that we have set up the construction, we can begin to get bounds on the degrees of curves. The idea is that if the image sheaf of the composition $f_b^* p_1^* M_1 \to f_b^* p_1^* M_d \to N_{Y_b/\PP^n}(\log D_b)$ has low degree, then it follows that $f_b^* p_1^*M_1$ must be unstable. The destabilizing subsheaf gives rise to a family of lines, which allows us to lift $Y_b$ to the universal line $V$. The goal of this section is to prove Theorem \ref{thm-associatedLine}.

We start by using Propositions \ref{prop-bounds} and \ref{prop-notsurj} to recover Pacienza and Rousseau's result \cite{PR-log}. This result is sharp in the sense that for $d\leq 2n$, a generic hypersurface has bicontact lines \cite{Zaidenberg} and hence is not algebraically hyperbolic.

\begin{Corollary}\label{cor-0}
	Suppose $d\geq 2n+1$ and $D\subseteq \PP^n$ is a very general hypersurface of degree $d$. Then $(\PP^n,D)$ is algebraically hyperbolic.
\end{Corollary}

\begin{proof}
	Let the hypersurface be parametrized by $b\in B$, and let $f_b:Y_b\rightarrow X_b=\PP^n$ be a curve of geometric genus $g$ and degree $e$ such that $Y_b\nsubseteq D_b=D$ and $i(Y_b, D)=i$. If $\alpha$ is not generically surjective, then Proposition \ref{prop-notsurj} implies that $$2g-2+i\geq (d-(2n-1))e.$$ If $\alpha$ is generically surjective, then we can apply Proposition \ref{prop-bounds} with $t\leq n-1$. So we obtain $\deg N_{Y_b/\PP^n}(\log D_b)\geq-(n-1)e,$ and so $2g-2+i\geq (d-2n)e.$
\end{proof}

What we aim to do now is to characterize the curves that violate (\ref{eqn-logah3}) and (\ref{eqn-logah2}). For the rest of this section assume $d=2n$. Due to Proposition \ref{prop-notsurj}, we will assume from now on that $\alpha$ is generically surjective. Let us extend Diagram (\ref{diag-3}) and then pull back via $\iota_b$ to obtain the following exact diagram. $\kK_4$ denotes the cokernel of $\ker\alpha\hookrightarrow T_\yY^{\ver}(-\log \dD)$. 

\begin{equation*}
	\begin{tikzcd}
		& &0\arrow[d] &0\arrow[d] & 0 \arrow[d] \\
		&0 \arrow[r] &\iota_b^*\ker\alpha \arrow[d] \arrow[r] &\iota_b^*T_\yY^{\ver}(-\log \dD) \arrow[d] \arrow[r] &\iota_b^*\kK_4 \arrow[r] \arrow[d] &0 \\
		&0 \arrow[r] &f_b^*p_1^*M_d \arrow[d, "\alpha_b"] \arrow[r] &f_b^*T_\xX^{\ver}(-\log \dD) \arrow[d] \arrow[r] &f_b^*p_1^*\OO_{\PP^n} \arrow[r] \arrow[d] &0\\
		& 0 \arrow[r] &\iota_b^*\text{Im }\alpha \arrow[d] \arrow[r] &N_{Y_b/\PP^n}(\log D_b) \arrow[d] \arrow[r] &\iota_b^*\coker\alpha \arrow[d] \arrow[r] &0\\		
		& &0 &0 &0
	\end{tikzcd}
\end{equation*}

As in the proof of Proposition \ref{prop-notsurj}, there is a generically surjective map 
\begin{equation}\label{eqn-surjM1}
	f_b^*p_1^*M_1^{\oplus t}\rightarrow \iota_b^* \text{Im }\alpha
\end{equation}
for some $t$. Explicitly, when restricted to each summand $f_b^*p_1^*M_1$, this map is induced by the multiplication map $M_1\xrightarrow{\cdot P} M_d$ for a generic polynomial $P\in H^0(\OO_{\PP^n}(d-1))$. From now on, we take $t$ to be the smallest number of $f_b^*p_1^*M_1$ summands needed to achieve (\ref{eqn-surjM1}). Let us denote the generic polynomials inducing these maps by $P_1,\ldots,P_t\in H^0(\OO_{\PP^n}(d-1))$. Let us also denote the following maps
\begin{equation*}
	m_i: \bigoplus_{j=1}^i f_b^*p_1^*M_1 \xrightarrow{\bigoplus_{j=1}^i \cdot P_j} f_b^*p_1^*M_d\xrightarrow{\alpha_b} \iota_b^*\text{Im }\alpha
\end{equation*}
for every $1\leq i\leq t$.



\begin{Proposition}\label{prop-m1equals2}
	Suppose that rank $m_1\geq 2$. Then, the curve $f_b:Y_b\rightarrow \PP^n$ satisfies 
    $$2g-2+i\geq (d-(2n-1))e.$$
    In particular, when $d=2n,$ the curve satisfies 
    $$2g-2+i\geq e.$$
\end{Proposition}

\begin{proof}
	Since rank $m_1\geq2$, we must have $t \leq n-2$. Hence, Proposition \ref{prop-bounds} implies that $\deg N_{Y_b/\PP^n}(\log D_b)\geq-(n-2)e,$ and so $2g-2+i\geq (d-(2n-1))e.$
\end{proof}

Due to the previous proposition, for the rest of this paper, we will focus on the case rank $m_1=1$. In this case, we can associate a special line to a generic point $(p,b)\in\yY$.

\begin{Lemma}\label{lem-2.10}
Suppose that rank $m_1=1$. Then, there exists a unique line $\ell:=\ell(p,b)\subseteq X_b=\PP^n$ through $p$ whose ideal $H^0(\OO_{\PP^n}(d)\otimes I_\ell)$ is contained in $\ker \alpha\vert_{(p,b)}$, except possibly when $n=2$ and the curve $f_b:Y_b\rightarrow X_b=\PP^2$ satisfies 
\begin{equation}\label{eqn-hn}
    2g-2+i\geq \frac{1}{2}e.
\end{equation}

When such a line $\ell(p,b)$ exists, we call it the \emph{associated line} for $(p,b)$.        
\end{Lemma}

\begin{proof}
	First consider the case $n\geq 3$. Since rank $N_{Y_b/\PP^n}(\log D_b)=n-1\geq 2,$ Lemma 2.1 in \cite{Clemens03} implies that $\ker m_1\vert_p$ is independent of the generic polynomial $P_1$ inducing the map $m_1$. Hence, the image of $$\ker m_1\vert_p \xrightarrow{-\otimes H^0(\OO_{\PP^n}(d-1))} M_d\vert_p,$$ which is the ideal $H^0(\OO_{\PP^n}(d)\otimes I_\ell)$ of some line $\ell=\ell(p,b)$ through $p$, is contained in $\ker \alpha\vert_p$.
    
    For $n=2$, we need a different argument to show that $\ker m_1\vert_p$ is independent of the generic polynomial $P_1$. For this argument, given some $P=P_1\in H^0(\OO_{\PP^n}(d-1)),$ we denote the map $m_1$ induced by $P$ as $m_P$, and define $Q_P$ to be the torsion-free part of the image of $m_P$. Since rank $f_b^*p_1^*M_1=2$, $\ker m_P$ and $Q_P$ must have rank one. Now we split into two cases depending on how their degrees compare.
    
    If $\deg \ker m_P\leq \deg Q_P,$ then $\deg Q_P\geq \frac{1}{2}\deg f_b^*p_1^*M_1= -\frac{1}{2}e$, which implies (\ref{eqn-hn}). 
    
    If $\deg \ker m_P> \deg Q_P,$ then $0\subseteq \ker m_P\subseteq f_b^*p_1^*M_1$ is the Harder-Narasimhan filtration, so $\ker m_P$ is independent of $P$, and we can now denote it as $\ker m_1$. The same argument as in the case $n\geq 3$ above implies the existence of some line $\ell=\ell(p,b)$ through $p$ whose ideal $H^0(\OO_{\PP^n}(d)\otimes I_\ell)$ is contained in $\ker \alpha\vert_p$.
    
    To prove uniqueness, suppose that there are two distinct lines through $p$ whose ideals are contained in $\ker \alpha\vert_p$. By a dimension count, we would then have $M_d\vert_p=\ker \alpha\vert_p$, which contradicts the assumption that $\alpha_b:f_b^*p_1^*M_d\rightarrow \iota_b^*\alpha$ is generically surjective. 
\end{proof}

For the rest of this section, we work in the boundary case $d=2n$ and assume that the general point of $\yY$ has an associated line.
We obtain the following natural rational map 
\begin{equation*}
	\rho:\yY \dashrightarrow B\times V \dashrightarrow \overline{U}_{0,d}^{\sym}	
\end{equation*}
which maps $(p,b)\in \yY$ to $(b, p, [\ell(p,b)])\in B \times V$ and then to the modulus in $\overline{U}_{0,d}^{\sym}$ whose first $d$ points are the intersection of $\ell(p,b)$ with $D_b$ and whose last point is $p$. 

\subsection{The $n=2$ case.}

In this subsection, we do some extra work for the $n=2$ case. If the general associated line for $(p,b)$ is a bitangent or flex line to $D_b$, then we see that $\yY$ is contained in the locus of bitangent and flex lines, which is compatible with Theorem \ref{thm-associatedLine}. Thus we may assume that the general associated line is not a bitangent or flex line for $D_b$, which implies that $\rho$ is well-defined at a general point of $\yY$.

\begin{Lemma} \label{lem-imRhoDim1}
    If $n = 2$, the image of $\yY$ under $\rho$ has dimension at most 1.
\end{Lemma} 

\begin{proof}
    Let $x,y,z$ be the coordinates on $\PP^2$, and let $p = V(y,z)$, $\ell = V(z)$. Let $\yY_p \subset \yY$ be the fiber over $p_1\circ f: \yY \rightarrow \PP^2$. Now consider the associated line map $\lambda: \yY_p \dashrightarrow \GG(1,2)$. Since $\yY$ is $GL(3)$-invariant, $\lambda$ maps onto the locus of lines through $p$. Let $\yY_{p,\ell}$ be the fiber of $\lambda$ over $[\ell]\in \GG(1,2)$, where $\ell:=\ell(p,b)$. By $GL(3)$-invariance of $\yY$, it is enough to show that $\yY_{p,\ell}$ has one-dimensional image under $\rho$.

    $\rho|_{\yY_{p,\ell}}$ can be decomposed as the composition of $\rho_1: \yY_{p,\ell} \to H^0(\ell,\OO_{\ell}(4))$ and $\rho_2: H^0(\ell,\OO_{\ell}(4)) \dashrightarrow \overline{U}_{0,4}^{\sym}$. Let $H_{bad} \subset H^0(\ell,\OO_{\ell}(4))$ be the locus where $\rho_2$ is not well-defined, namely, the locus of polynomials with at least a triple root union the locus of polynomials with a double root at $p$. If the image of $\rho_1$ is contained in $H_{bad}$, then $\rho$ is not well-defined, so we can assume that the image of $\rho_1$ is not contained in $H_{bad}$.

    By Lemma \ref{lem-2.10}, $H^0(\PP^2,I_{\ell}(4))$ is contained in $T_\yY^{\ver}(-\log \dD)\vert_{(p,b)}$. Since $\yY_{p,\ell}$ has codimension at most 1 in $\yY_p$, the relative tangent space to $\rho_1$ at $(p,b)$ will consist of a codimension at most 1 space in the affine space \linebreak $F_b + H^0(\PP^2,I_{\ell}(4))$, where $F_b$ is the polynomial defining the quartic plane curve $D_b$. Then, the result follows by a dimension count: the dimension of $\yY_{p,\ell}$ is $\dim \yY - 3 = \dim B - 2$. The dimension of the space of degree 4 polynomials restricting to a particular polynomial on $\ell$ is $\dim H^0(\PP^2,I_{\ell}(4)) = \dim B - 5$, and the dimension of the fibers of $\rho_2: H^0(\ell, \OO_{\ell}(4)) \dashrightarrow \overline{U}_{0,4}^{\sym}$ is 3 away from $H_{bad}$. Thus, the dimension of the image of $\rho|_{\yY_{p,\ell}}$ is at most 
    \[ \dim \yY_{p,\ell} - (\dim H^0(\PP^2,I_{\ell}(4)) - 1) - 3 \]
    \[ = \dim B - 2 - (\dim B - 6) - 3 = 1 .\]
\end{proof}

For the proof of Lemma \ref{lem-fixedModulus}, we need a complete description of the tangent space $T_{\yY}^{\ver}(-\log \dD)|_{(p,b)}$. We may consider the splitting types of the restriction of $T_{\PP^2}(-\log D_b)$ to lines $\ell$ in $\PP^2$. By the Grauert--Mulich Theorem (see \cite[Theorem 2.1.4]{OkonekSchneiderSpindler} or \cite{PatelRiedlTseng}), the restriction $T_{\PP^2}(-\log D_b)|_{\ell}$ to a general line $\ell \subset \PP^2$ has the balanced splitting type $\OO \oplus \OO(-1)$. When the restriction $T_{\PP^2}(-\log D_b)|_{\ell}$ has the unbalanced splitting type $\OO(-2) \oplus \OO(1)$, we say that $\ell\subset \PP^2$ is a \emph{jumping line} for $T_{\PP^2}(-\log D_b)$. Lemma \ref{lem-Ominus2splitting} shows that there are only finitely many jumping lines for a general $D_b$, and Lemma \ref{lem-P2jumping} shows that if the general associated line is not a jumping line for $T_{\PP^2}(-\log D_b)$, then the tangent space $T_{\yY}^{\ver}(-\log \dD)|_{(p,b)}$ consists of $H^0(\PP^2,I_{\ell}(4))$ plus the vertical tangent directions to the $GL(3)$-orbit of $(p,b)$.

\begin{Lemma}
\label{lem-Ominus2splitting}
For $D$ a general quartic curve, a general line $\ell$ will have $T_{\PP^2}(-\log D)|_{\ell} \cong \OO \oplus \OO(-1)$. The space of lines in $\PP^2$ with $T_{\PP^2}(- \log D)|_\ell) \cong \OO(1) \oplus \OO(-2)$ is finite.
\end{Lemma}
\begin{proof}
The first statement follows from the Grauert-Mulich Theorem, see for instance Theorem 2.1.4 of \cite{OkonekSchneiderSpindler} or \cite{PatelRiedlTseng} for more details. To see the second statement, consider the locus $W$ in $\PP^2 \times H^0(\PP^2,\OO_{\PP^2}(4))$ of pairs $(\ell, f)$ such that $T_{\PP^2}(-\log V(f))|_\ell \cong \OO(1) \oplus \OO(-2)$. We claim that $W$ has codimension 2, from which it follows that the projection of $W$ onto $H^0(\PP^2, \OO_{\PP^2}(4))$ has finite (or possibly empty) general fiber.

To show this, we consider a line $\ell$ in $\PP^2$ and argue that the space of $f$ such that $T_{\PP^2}(-\log V(f))|_\ell \cong \OO(1) \oplus \OO(-2)$ has codimension 2 in $H^0(\PP^2,\OO_{\PP^2}(4))$. This will not depend on the choice of coordinates, so assume that $\ell$ is the line $V(x)$. Consider the following diagram.
\begin{equation*}
	\begin{tikzcd}
		& & &0\arrow[d] &0\arrow[d] \\
		& & & \OO_\ell \arrow[d] \arrow[r] & \OO_\ell  \arrow[d] \\
		&0 \arrow[r] & T_{\PP^2}(-\log D)|_\ell \arrow[d, "\simeq"] \arrow[r] & \OO_{\ell}(1)^3 \arrow[d] \arrow[r, "\alpha"] & \OO_{\ell}(4) \arrow[r] \arrow[d] &0 \\
		&0 \arrow[r] & T_{\PP^2}(-\log D)|_\ell \arrow[r] & T_{\PP^2}|_{\ell} \arrow[d] \arrow[r] & \OO_{D \cap \ell}(D) \arrow[d] \arrow[r] &0\\	
		& & & 0 &0 
	\end{tikzcd}
\end{equation*}

Here, $\alpha$ is the map of partial derivatives $(f_x, f_y, f_z)$ restricted to $\ell$. For $T_{\PP^2}(-\log D)|_\ell$ to have an $\OO(1)$ summand, we see that some linear combination of the partials must be 0. Write $f$ as $f = x g(x,y,z) + h(y,z)$. Then the coordinates of the map $(f_x, f_y, f_z)$ restricted to $\ell$ will be $(g, h_y, h_z)$. There is a 2-dimensional choice for which linear combination $af_x+bf_y+cf_z$ is 0 (because the vanishing is invariant under scaling), and so it remains to show that a fixed linear combination, the codimension of $f$ such that $af_x+bf_y+cf_z$ vanishes upon restricting to $\ell$ is 4. If $b = c= 0$, we see that this is equivalent to the vanishing of $g|_{\ell}$, which is indeed four conditions: the vanishing of the $xy^3, xy^2z, xyz^2$ and $xz^3$ terms. If one of $b$ and $c$ is nonzero, we can select our coordinate $y$ so that $f_y$ is zero. This implies that $h_y$ is zero, which implies the vanishing of the $y^4, y^3z, y^2z^2$, and $yz^3$ terms, which is again four conditions. The result follows.

\end{proof}

\begin{Lemma}\label{lem-P2jumping}
    If the general associated line is not a jumping line for \linebreak $T_{\PP^2}(-\log D_b)$, then the space $\ker \alpha|_{(p,b)} \subset T_{\yY}^{\ver}(-\log \dD)|_{(p,b)}$ is generated by $H^0(\PP^2,I_{\ell}(4))$ and $M_1|_p \cdot \Jac(F_b)$. Hence, $T_{\yY}^{\ver}(-\log \dD)|_{(p,b)}$ is generated by \linebreak $H^0(\PP^2,I_{\ell}(4))$, $M_1|_p\cdot \Jac(F_b)$ and $F_b$.
\end{Lemma}

\begin{proof}
    Choose coordinates $x,y,z$ on $\PP^2$ so that $p = V(y,z)$, $\ell = V(z)$. For convenience, write $F = F_b$. Since we are assuming that $\alpha$ is generically surjective, we work in the case that $\ker \alpha|_{(p,b)}$ has codimension 2 in $T_{\xX}^{\ver}(-\log \dD)|_{(p,b)} = H^0(\PP^2,\OO_{\PP^2}(4))$. 
    We know that $H^0(\PP^2,I_{\ell}(4))$ is contained in $\ker \alpha|_{(p,b)}$, and by $GL(3)$-invariance of $\yY$, it follows that \linebreak $M_1|_p\cdot \Jac (F) = (yF_x, yF_y, yF_z, zF_x, zF_y, zF_z)$ are all contained in $\ker \alpha|_{(p,b)}$. Thus, it remains to show that $\ker \alpha|_{(p,b)}$ is generated by these elements. This will follow if we can show that $H^0(\PP^2,I_{\ell}(4)) + M_1|_p\cdot \Jac (F)$ has codimension 2 in $H^0(\PP^2,\OO_{\PP^2}(4))$.

    
    Suppose that $\ell$ is not a jumping line for $T_{\PP^2}(-\log D_b)$. \cite[Theorem 10]{Marangone} characterizes the non-jumping lines for any vector bundle $\mathcal{E}$ of rank two over $\PP^2$ as exactly the weak Lefschetz elements for its first cohomology module $H^1_*(\PP^2,\mathcal{E}).$ Let $A=\CC[x,y,z]/Jac(F)$ denote the first cohomology module for $T_{\PP^2}(-\log D_b)$, and let $A_i$ be the $i$-th graded piece of $A$. Then, it follows that the multiplication map by $z$ is surjective as a map from $A_3 \to A_4$. Thus, we see that $H^0(\PP^2,I_{\ell}(4))$ and $xF_x, xF_y, xF_z, yF_x, yF_y, yF_z$ generate $H^0(\PP^2,\OO_{\PP^2}(4))$. 
     
     We now work modulo $H^0(\PP^2,I_{\ell}(4))$. Since the space $H^0(\ell,\OO_{\ell}(4)) = H^0(\PP^2,\OO_{\PP^2}(4)) / H^0(\PP^2,I_{\ell}(4)) $ has dimension 5, the six polynomials $xF_x, xF_y, $ $ xF_z, yF_x, yF_y, yF_z$ must admit a linear relation modulo $H^0(\PP^2,I_{\ell}(4))$. If the relation does not involve the $xF_j$ and only involves $yF_x, yF_y, yF_z$, then it implies a relation among $F_x, F_y, F_z$, and hence, a relation among $xF_x, xF_y, xF_z$. Thus, there must be some relation among $xF_x, xF_y, xF_z, yF_x, yF_y, yF_z$ that involves at least one of $xF_x, xF_y$, and $xF_z$. In particular, it follows that $M_1|_p\cdot \Jac (F)$ has codimension at most 2 in $H^0(\PP^2,\OO_{\PP^2}(4)) / H^0(\PP^2,I_{\ell}(4))$ as required.

     For general $b$, note that $F_b$ does not vanish at $p$, so $F_b$ is not in $\ker \alpha|_{(p,b)}$, but it clearly is in $T_{\yY}^{\ver}(-\log \dD)|_{(p,b)}$. The result follows.
\end{proof}

\subsection{The image of $\rho$.} Lemma \ref{lem-fixedModulus} shows that $\yY$ has to be a family of curves with a special geometric property. In \S\ref{sec-fixedModulus}, we exploit this property to identify the degree $e$ and genus $g$ of the curves in such a family in the case of a very general quartic curve in $\PP^2$. 

\begin{Lemma}\label{lem-fixedModulus}
	Suppose that $d=2n$, and that the associated line is defined at a general point $(p,b) \in \yY$.
	Then, the map $\rho:\yY \dashrightarrow \overline{U}_{0,d}^{\sym}$ is constant.
\end{Lemma}

\begin{proof}
    As in the proof of Lemma \ref{lem-imRhoDim1}, let $\yY_p \subset \yY$ be the fiber of $p_1\circ f: \yY \rightarrow \PP^n$ over $p$. Since $\yY$ is $GL(n+1)$-invariant, the associated line map $\lambda: \yY_p \dashrightarrow \GG(1,n)$ maps onto the locus of lines through $p$. Let $\yY_{p,\ell}$ be the fiber of $\lambda$ over $[\ell]\in \GG(1,n)$, where $\ell:=\ell(p,b)$. It is enough to show that $\yY_{p,\ell}$ has 0-dimensional image under $\rho$. We can factor the restriction map $\rho_\ell: \yY_{p,\ell} \dashrightarrow \overline{U}_{0,d}^{\sym}$ as 
    \begin{center}
        \begin{tikzcd} 
            &\yY_{p,\ell} \arrow[rr, dashed, "\rho_\ell"] \arrow[dr, "\rho_1"] & &\overline{U}_{0,d}^{\sym}\\
            & &H^0(\ell,\OO_\ell(d)) \arrow[ur, dashed, "\rho_2"]
        \end{tikzcd}.
    \end{center}

    The case $n \geq 3$ follows directly from an argument of Voisin \cite[Lemma 3]{voisinCorrection}, which we summarize here for the reader's convenience. We have a distribution $\mathcal{I} \subset T_{\yY}^{\ver}$ on $\yY$, whose fiber over a general point $(p,b)$ of $\yY$ is precisely $H^0(\OO_{\PP^n}(d) \otimes I_{\ell})$. We can understand the Lie bracket map $\psi: \bigwedge^2 \mathcal{I} \to T_{\yY}^{\ver}/\mathcal{I}$ at the point $(p,b)$ explicitly in coordinates, in terms of the differential of the associated line map, which we denote by $\phi: T_{\yY}^{\ver}\vert_{(p,b)} \to T_{\ell} \,\GG(1,n)_p$. Here $\GG(1,n)_p$ is the space of lines in $\PP^n$ passing through $p$, and $ T_{\ell} \,\GG(1,n)_p \simeq H^0(\ell, N_{\ell / \PP^n}(-p))$. A local calculation shows that $\phi$ vanishes on $\mathcal{I}^2$, so the map $\psi$ at $(p,b)$ descends to a map
    \[ \psi: \bigwedge^2(H^0(\OO_{\ell}(d)(-p)) \otimes K^{*}) \to H^0(\OO_{\ell}(d)(-p)) , \]
    where $K = H^0(\ell, N_{\ell / \PP^n}(-p))$ and the map is given by $\psi(A \wedge B) = A\cdot\phi(B) - B\cdot\phi(A)$, where $\cdot$ comes from the natural contraction pairing of $K^*$ with $K$. Using \cite[Lemma 4]{voisinCorrection}, it follows that if $\phi$ is nontrivial on $\mathcal I$, the image of $\psi$ must have codimension at most 1 in $H^0(\OO_{\ell}(d)(-p))$. This contradicts our choice of $\yY$ being codimension $n-1\geq 2$ in $\xX$, so it follows that $\phi$ must be trivial on $\mathcal I$. As in \cite{voisinCorrection}, it follows that the vertical tangent space $T_{\rho_1}$ of the map $\rho_1$ contains $H^0(\OO_{\PP^n}(d)\otimes I_\ell)$. By the $GL(n+1)$-invariance of $\yY$, the image of $\rho_1$ is closed under the $GL(2)$-actions that fix $p$. The tangent directions corresponding to these actions, which span a space of dimension 3, are contained in the vertical tangent space $T_{\rho_2}$ of $\rho_2$. So, we have the following bound on the dimension of the image of $\rho_\ell$.
    \begin{align*}
        \dim \text{Im }\rho_\ell 
        \leq& \dim \yY_{p,\ell}-h^0(\OO_{\PP^n}(d)\otimes I_\ell)-3\\
        =& (\dim \yY_{p}-(n-1))-h^0(\OO_{\PP^n}(d)\otimes I_\ell)-3\\
        =& \dim \iota_b^*T_\yY^{\ver}(-\log \dD)\vert_p -(n-1)-h^0(\OO_{\PP^n}(d)\otimes I_\ell)-3\\
        =&(\dim f_b^*T_\xX^{\ver}(-\log \dD)\vert_p-\dim N_{Y_b/\PP^n}(\log D_b)\vert_p)\\
    &\quad -(n-1)-h^0(\OO_{\PP^n}(d)\otimes I_\ell)-3\\
        =&\dim f_b^*T_\xX^{\ver}(-\log D)\vert_p-h^0(\OO_{\PP^n}(d)\otimes I_\ell)\\
    &\quad -\dim N_{Y_b/\PP^n}(\log D_b)\vert_p-(n+2)\\
        =&(d+1)-(n-1)-(n+2)\\
        =&d-2n=0.
    \end{align*}
    
    The other case is $n=2$. We have shown in Lemma \ref{lem-imRhoDim1} that $\rho$ has an image of dimension at most one. If the image is 0-dimensional, we are done, so suppose not. We will show that the image of $\yY_p$ in $\overline{U}_{0,4}^{\sym}$ is 0-dimensional. The lifiting $\yY \to B \times V$ restricts to a lifting $\yY_p \to B \times \GG(1,2)_p$, where $\GG(1,2)_p$ is the space of lines through $p$. Let $Z$ be the preimage of $\rho(\yY_p)$ in $B \times \GG(1,2)_p$. We see that $Z$ has codimension 1 in $B \times \GG(1,2)_p$, so that $Z$ maps generically finitely onto $\xX_p$, giving an isomorphism of tangent spaces at a general point. The relative tangent space of $Z \to \overline{U}_{0,4}^{\sym}$ is $H^0(\PP^2,I_{\ell}(4)) + M_1|_p\cdot \Jac(F) + F$, coming from the polynomials with fixed restriction to the line together with the automorphisms of $\PP^2$ fixing $p$ that move the line. By Lemma \ref{lem-P2jumping}, this is the same as $T_{\yY}^{\ver}(-\log \dD)|_{(p,b)}$, so it follows that $\yY$ must be a fiber of the map to $\overline{U}_{0,4}^{\sym}$, as required. 
\end{proof}

\begin{proof}[Proof of Theorem \ref{thm-associatedLine}]
This is simply piecing everything together. If for a general hypersurface $D$ we can find a curve violating (\ref{eqn-logah3}) or (\ref{eqn-logah2}), then we can spread the curve out into a universal family as in Proposition \ref{prop-mainSetupProp}. We wish to understand the log normal sheaf. Using Proposition \ref{prop-mainLazarsfeldProp}, we see that it admits a surjection from some copies of pullbacks of $M_1$. By Proposition \ref{prop-m1equals2}, we see that the rank $m_1$ of the image of the first $M_1$ must be 1. By Lemma \ref{lem-2.10}, it follows that we can lift the curve $C$ to $V$. By Lemma \ref{lem-fixedModulus}, it follows that the lifting of $C$ to $V$ must lie in a fiber of $\theta: V \dashrightarrow \overline{U}_{0,d}^{\sym}$ or the undefined locus of $\theta$. The result follows.
\end{proof}

\section{Curves with fixed modulus}
\label{sec-fixedModulus}

We now turn our attention to an exact characterization of $S$ when $n=2$. In this section, we assume $n = 2$, $d=4$. By Theorem \ref{thm-varmain}, we know that the only curves that can have $2g-2 < \frac{1}{2} \deg C$ are those that lift to $V$ via the associated line map and are in the fibers (or undefined locus) of the rational map to $\overline{U}_{0,4}^{\sym}$. Since we already understand the undefined locus of the map, it remains to understand the fibers of $\theta$. 
We will use a result about the log tangent sheaf and the space of lines given by Lemma \ref{lem-Ominus2splitting}.

\subsection{Local analysis of curves of constant modulus}

Recall that
$$
(p_1,p_2,p_3,p_4)\in M_{0,4}
$$
for four distinct points $p_i\in \PP^1$ is parameterized by its cross ratio:
$$
\lambda = \frac{(p_1 - p_2)(p_3 - p_4)}{(p_1 - p_3)(p_2 - p_4)}
$$
The symmetric group on $p_i$ acts on $\lambda$ by sending it to
$$
1 - \lambda,\ \frac{1}{\lambda},\ \frac{1}{1-\lambda},\ \frac{\lambda - 1}{\lambda},\ \frac{\lambda}{\lambda - 1} 
$$
Thus $M_{0,4}^{\sym}$ is parameterized by the $j$-invariant
$$
j(\lambda) = \frac{2^8 (\lambda^2 - \lambda + 1)^3}{\lambda^2 (\lambda - 1)^2}
$$
with $\lambda$ the cross ratio of $(p_1,p_2,p_3,p_4)\in M_{0,4}^{\sym}$.

The following lemma allows us to understand the family of associated lines.

\begin{Lemma}\label{LANGLOGAH2LEMFIXMODULI}
Given a smooth quartic curve $D\subset \PP^2$, let $C$ be an integral curve in $\GG(1,2) = (\PP^2)^\vee$ such that the rational map $\gamma: \GG(1,2)\dashrightarrow M_{0,4}^{\sym}$ sending $L\in \GG(1,2)$ to $L\cap D$ is constant at general points of $C$. Let $j$ be the $j$-invariant of the image $\gamma(C)$.
Then for $D$ general we have the following:
\begin{itemize}
\item $C$ passes through the $24$ flex lines to $D$;
\item if $j = j(2)$, then $C$ is a curve of degree $6$;
\item if $j = j(-e^{2\pi i/3})$, then $C$ is a curve of degree $4$;
\item if $j\ne j(2), j(-e^{2\pi i/3})$, then $C$ is a curve of degree $12$ with $24$ cusps at the $24$ flex lines; in addition, for $j$ general, $C$ is smooth outside of the $24$ cusps;
\item for each $j$, there exists a unique integral curve $C$ whose image $\gamma(C)$ is a point with $j$-invariant $j$;
\item two such curves $C_1$ and $C_2$ meet only at flex lines $L$ to $D$ with multiplicities
$$
(C_1. C_2)_{L} = \frac{1}{24} (C_1 . C_2).
$$
\end{itemize}
\end{Lemma}

\begin{proof}
Since $C$ meets the dual curve of $D$, there is at least one line $L\in C$ that fails to meet $D$ transversely. Clearly, such $L$ cannot meet $D$ at a point with multiplicity $2$; otherwise, the image of $C$ will meet the boundary point of $\overline{M}_{0,4}^{\sym}$. So for $D$ general, such $L$ can only be a flex line of $D$. It is also clear that two such distinct curves $C_1$ and $C_2$ can only meet at these flex lines.

The map $\gamma: \GG(1,2)\dashrightarrow M_{0,4}^{\sym}$ is smooth at a line $L\in \GG(1,2)$ meeting $D$ transversely if and only if we have a surjection
$$
\begin{tikzcd}
H^0(N_{L/\PP^2}) \ar[two heads]{r} & H^1(T_L(-\log D)) 
\end{tikzcd}
$$
induced by the exact sequence
$$
\begin{tikzcd}
0 \ar{r} & T_L(-\log D) \ar{r} & T_{\PP^2}(-\log D) \Big|_L \ar{r} & N_{L/\PP^2}(\log D) = N_{L/ \PP^2} \ar{r} & 0.
\end{tikzcd}
$$
This happens if and only if $H^1(L, T_{\PP^2}(-\log D)) = 0$, or equivalently,
$$
T_{\PP^2}(-\log D)\Big|_L = \OO_L(-1) \oplus \OO_L.
$$

By Lemma \ref{lem-Ominus2splitting} there are only finitely many such lines $L$. In conclusion, $\gamma$ is smooth outside of a finite set of points plus the dual curve of $D$. Therefore, for $j$ general, $C$ is smooth outside of the $24$ flex lines. Let us figure out the type of singularity of $C$ at a flex line $L$.

For simplicity, suppose that $C$ passes through the flex line $L = \{y = 0\}$ of $D$, where $(x,y)$ are the affine coordinates of $\PP^2$. In these coordinates, $D$ is given by the affine equation
$$
y g(x,y) + x^3 + c x^4 = 0,
$$
and we can scale so that $g(0,0) = 1$.
Let $(s,t)$ be local coordinates of $\GG(1,2)$ at $L$, so that a deformation of $L$ in $\GG(1,2)$ is given by
$$
y = sx + t .
$$
Then $y = sx + t$ meets $D$ at the points
$$
x^3 + sx + t + cx^4 + (sx + t)(g(x, sx + t) - 1) = 0 .
$$
The above equation has four roots $r_i = r_i(s,t)$ for $i=1,2,3,4$, where $r_4 \in \CC[[s,t]]$ and $r_1, r_2, r_3$ are algebraic functions in $s$ and $t$, defined in some finite extension of $\BC[[s,t]]$. They satisfy
$$
r_1(0,0) = r_2(0,0) = r_3(0,0) = 0 \text{ and } r_4(0,0) = -\frac{1}c.
$$
The limit of the cross ratio of these four points is
$$
\lambda = \lim_{(s,t)\to (0,0)} \frac{(r_1 - r_2)(r_3 - r_4)}{(r_1 - r_3)(r_2 - r_4)} = \lim_{(s,t)\to (0,0)} \frac{r_1 - r_2}{r_1 - r_3}
$$
Hence $j = j(\lambda)$.

Given the $j$-invariant $j$, there exists a pair $(a, b)$ such that if $r_1,r_2,r_3$ are the three roots of the equation
$$
x^3 + a x + b = 0,
$$
then
$$
j = j\left(\frac{r_1 - r_2}{r_1 - r_3}\right).
$$
The ratio $a^3/b^2$ uniquely characterizes the pairs $(a,b)$ with this property, so we may regard $(a,b)$ as a point in the corresponding weighted projective line $\PP(2,3)$.

When $a = 0$ and $b\ne 0$, $j = j(-e^{2\pi i/3})$. When $a \ne 0$ and $b = 0$, $j = j(2)$. Suppose that $j\ne j(2), j(-e^{2\pi i/3})$. Then the corresponding $(a,b)$ satisfies $ab \ne 0$.

For some fixed local branch $C$ of a component of a fiber of $\gamma$, let
$$
\left\{
\begin{aligned}
s &= a u^m\\
t &= b u^n + O(u^{n+1})
\end{aligned}\right.
$$
be the local normalization of $C$ at $(0,0)$. Let us consider the roots of
$$
x^3 + (a u^m x + b u^n  + O(u^{n+1})) g(x, a u^m x + b u^n + O(u^{n+1})) + c x^4 = 0.
$$

If $3m > 2n$, then
$$ \lim_{u\to 0} \frac{r_1 - r_2}{r_1 - r_3} = -e^{2\pi i/3} $$
after some permuation of $r_1, r_2, r_3$. This is impossible because we assume that $j\ne j(-e^{2\pi i/3})$.

Similarly, if $3m < 2n$, then
$$ \lim_{u\to 0} \frac{r_1 - r_2}{r_1 - r_3} = 2 $$
after some permuation of $r_1, r_2, r_3$. This is impossible because we assume that $j\ne j(2)$.

So we necessarily have $3m = 2n$. Let us rewrite the normalization as
$$
\left\{
\begin{aligned}
s &= a u^{2l}\\
t &= b u^{3l} (1 + v)
\end{aligned}\right.
$$
where $v = O(u)$ is a function of $u$. Let us consider the roots of
$$
x^3 + (a u^{2l} x + b u^{3l}(1+v)) g(x, a u^{2l} x + b u^{3l} (1+v)) + c x^4 = 0
$$
in $\BC[[u, v]]$. Clearly, all $r_1, r_2, r_3, r_4$ are functions in $u^l$ and $v$. Furthermore,
$$
\left\{
\begin{aligned}
r_1 &= u^l (c_1 + d_1v + O(u^l, v^2))\\
r_2 &= u^l (c_2 + d_2v + O(u^l, v^2))\\
r_3 &= u^l (c_3 + d_3v + O(u^l, v^2))\\
r_4 &= -\frac{1}c + O(u^{2l}, u^{3l}(1+v))
\end{aligned}\right.
$$
where $c_1, c_2, c_3$ are three roots of $x^3 + ax + b = 0$ and $d_1, d_2, d_3$ are constants satisfying the equations
$$
\left\{
\begin{aligned}
d_1 + d_2 + d_3 &= 0\\
c_1d_1 + c_2 d_2 + c_3 d_3 &= 0\\
c_3 d_1d_2 + c_1 d_2d_3 + c_2 d_3d_1 &= b
\end{aligned}\right.
$$
Since $r_1,r_2, r_3, r_4$ have constant $j$-invariant,
$$
\frac{(r_1 - r_2)(r_3 - r_4)}{(r_1 - r_3)(r_2 - r_4)} \equiv \frac{c_1 - c_2}{c_1 - c_3}
$$
The above equation, viewed as an equation in $\CC[[u,v]]$, will have non-vanishing coefficient of $v$ provided that
\begin{equation}\label{LANGLOGAH2E001}
\det \begin{bmatrix}
c_1 - c_2 & d_1 - d_2\\
c_1 - c_3 & d_1 - d_3
\end{bmatrix} \ne 0
\end{equation}
which holds provided that $j\ne j(-e^{2\pi i/3})$. Therefore, we may write
$$
v = v(u^l)
$$
for a function $v(w)\in \CC[[w]]$ satisfying $v(0) = 0$. So the normalization of $C$ is
$$
\left\{
\begin{aligned}
s &= a u^{2l}\\
t &= b u^{3l} (1 + v(u^l))
\end{aligned}\right.
$$
Since the normalization of $C$ is birational onto its image, we must have $l=1$. In conclusion, the normalization of $C$ is given by
\begin{equation}\label{LANGLOGAH2E000}
\left\{
\begin{aligned}
s &= a u^2\\
t &= b u^3(1 + v(u))
\end{aligned}\right.
\end{equation}
for some $v(u) \in \CC[[u]]$ satisfying $v(0) = 0$. 
That is, when $ab \ne 0$, $C$ has a cusp at $L$.
It follows from the description above that $v(u)$ is uniquely determined by $(a,b)$.

Furthermore, when $b\to 0$, i.e., $j\to j(2)$, the flat limit of $C$ is simply the image of
$(s,t) = (au^2, 0)$, i.e., locally a nonreduced curve supported on a smooth curve with multiplicity $2$. This is due to the fact that \eqref{LANGLOGAH2E001} still holds for $j=j(2)$.

Let us now prove that $C$ passes through all $24$ flex lines and $\deg C = 12$ for $j$ general.
To be more precise, we can extend $\gamma$ to a rational map
$$
\begin{tikzcd}
\GG(1,2)\times |\OO_{\PP^2}(4)| \ar[dashed]{r}{\gamma} & M_{0,4}^{\sym} \times |\OO_{\PP^2}(4)|
\end{tikzcd}
$$
in an obvious way. We claim that the closure $\overline{\gamma^{-1}(j, D)}$ of $\gamma^{-1}(j, D)$ for $(j,D)$ general is an integral curve $C$ of degree $12$ passing through all $24$ flex lines to $D$. Note that $\overline{\theta^{-1}(j, D)}$, a priori, is not necessarily irreducible.

Let $F$ be the set of $24$ flex lines to $D$.
Let $C_1$ be an irreducible component of $\overline{\gamma^{-1}(j_1, D)}$ for $(j_1,D)$ general. Suppose that $C_1$ passes through a subset $F_1$ of $F$. 
Due to the uniqueness of $v(u)$ in \eqref{LANGLOGAH2E000}, $C_1$ has only one local branch at each line $L\in F_1$.

By \cite{Har79}, the monodromy group on the $24$ flex lines of $D$ is the full symmetric group as $D$ varies. So there is a component $C_1^\sigma$ of $\overline{\gamma^{-1}(j_1, D)}$ passing through $\sigma(F_1)$ for every permutation $\sigma$ of $F$.  

Let $C_2$ be an irreducible component of $\overline{\gamma^{-1}(j_2, D)}$ for $(j_2,D)$ general and $j_2\ne j_1$. Suppose that $C_2$ passes through a subset $F_2$ of $F$. Then $C_1$ and $C_2$ meet only at $L\in F_1\cap F_2$ and $(C_1.C_2)_L = 6$ by \eqref{LANGLOGAH2E000}. Therefore,
$$
C_1 . C_2 = 6 |F_1\cap F_2|
$$
For the same reason, we have
$$
C_1^{\sigma} . C_2 = 6 |\sigma(F_1)\cap F_2|.
$$
Since $\deg C_1 = \deg C_1^{\sigma}$, we conclude that
$$
|F_1\cap F_2| = |\sigma(F_1)\cap F_2|.
$$
This holds for every permutation $\sigma$ on $F$. So we necessarily have $F_1 = F$ or $F_2 = F$.
And since $j_1$ and $j_2$ are general, this implies that
$$
F_1 = F_2 = F \text{ and } \deg C_1 = \deg C_2 = 12.
$$
That is, every component of $\overline{\gamma^{-1}(j, D)}$ is an integral curve of degree $12$ with cusps at $L\in F$ for $(j, D)$ general.

Also $\overline{\gamma^{-1}(j, D)}$ cannot have more than one irreducible component. Otherwise, if $C$ and $C'$ are two distinct irreducible components of $\overline{\gamma^{-1}(j, D)}$, both curves are locally given by \eqref{LANGLOGAH2E000} at each flex line $L\in F$ with the same $a$ and $b$. Hence
$(C.C')_L > 6$
and it follows that $C.C' > 144$, which contradicts the fact that $\deg C = \deg C' = 12$.

This proves that for $(j,D)$ general, $\overline{\gamma^{-1}(j, D)}$ is an integral curve of degree $12$ with cusps at the $24$ flex lines to $D$ and smooth elsewhere. Next, let us prove that this is also true for all $j \ne j(2), j(-e^{2\pi i/3})$ and $D$ general.

Let $\Gamma = \overline{\gamma^{-1}(j, D)}$ for some $j \ne j(2), j(-e^{2\pi i/3})$ and $D$ general, where we consider $\gamma^{-1}(j, D)$ as the scheme-theoretical preimage of $(j,D)$. Since $\Gamma$ is the flat limit of $C = \overline{\gamma^{-1}(j, D)}$ for $(j, D)$ general, $\Gamma$ has degree $12$ and
$(\Gamma . C)_L = 6$ for all flex lines $L\in F$. Every irreducible component of $\Gamma$ has a cusp at $L$ if it passes through $L\in F$. So for every $L\in F$, there exists a unique irreducible component of $\Gamma$ passing through $L$. Therefore, $\Gamma$ must be reduced and irreducible; otherwise, two components $G_1$ and $G_2$ of $\Gamma$ must be disjoint, which is impossible. 

This proves that $\overline{\gamma^{-1}(j, D)}$ is an integral curve of degree $12$ with cusps at the $24$ flex lines to $D$ for all
$j \ne j(2), j(-e^{2\pi i/3})$
and $D$ general.

Finally, let us deal with the cases $j = j(2), j(-e^{2\pi i/3})$. Again, we let $\Gamma = \overline{\gamma^{-1}(j, D)}$. We still have $\deg \Gamma = 12$ and $(\Gamma . C)_L = 6$ for all $L\in F$.

If $j=j(2)$, then $b=0$. Since the flat limit of \eqref{LANGLOGAH2E000} as $b\to 0$ exists, $\Gamma$ is nonreduced with multiplicty $2$ at each $L\in F$. So again, for every $L\in F$, there exists a unique irreducible component of $\Gamma$ passing through $L$. Therefore, $\Gamma$ has a unique irreducible component and hence $\Gamma = 2G$ for an integral curve $G$ of degree $6$.

If $j=j(-e^{2\pi i/3})$, then $a=0$. By \eqref{LANGLOGAH2E000}, $\Gamma$ either is nonreduced with multiplicty $3$ or has a singularity of multiplicity $3$ at each $L\in F$. So again, for every $L\in F$, there exists a unique irreducible component of $\Gamma$ passing through $L$. Therefore, $\Gamma$ has a unique irreducible component; either $\Gamma$ is reduced and irreducible or $\Gamma = 3G$ for an integral curve $G$ of degree $4$. However, in this case, the former is impossible by comparing the arithmetic genus of $\Gamma$ and the $\delta$-invariants of $\Gamma$ at $L\in F$: $p_a(\Gamma) = 55 < 3\cdot 24$.
\end{proof}

\subsection{Global analysis}
Now that we understand the singularities of the curves corresponding to the fibers of $\theta$, we can work out their cohomology classes on $V$, and from that can get a bound on the degree of the log normal bundle.

\begin{Proposition}\label{prop-fibersofphiclasses}
Let $V\subset \PP^2\times \GG(1,2)$ be the universal family of lines in $\PP^2$.
Given a smooth quartic curve $D\subset \PP^2$, 
let $C$ be an integral curve in $V$ such that the rational map $\theta: V\dashrightarrow \overline{U}_{0,4}^{\sym}$ sending $(p,L)$ to $(p, L\cap D)$ is constant at general points of $C$.
Then $C$ has class $(18,12)$, $(18, 6)$, $(18, 4)$, $(36,12)$ or $(54,12)$ in $V \subset \PP^2 \times \PP^2$ for $D$ general.
\end{Proposition}

\begin{proof}
By Lemma \ref{LANGLOGAH2LEMFIXMODULI}, $\deg \eta(C) = 12, 6$ or $4$, where $\eta: V\to \GG(1,2)$ is the projection $V\to \GG(1,2)$. Let $G = \eta(C)$, $S = \eta^{-1}(G)$ and $\nu: \widehat{S} \to V$ be the normalization of $S$. Obviously, $\widehat{S}$ is a $\PP^1$-bundle over the normalization $\widehat{G}$ of $G$.

Let $\widehat{D} = \nu^* \pi^* D$, where $\pi$ is the projection $V\to \PP^2$. Clearly, $\widehat{D}$ is a multisection of $\widehat{S}/\widehat{G}$ of degree $4$. The map $\widehat{D}\to \widehat{G}$ is unramified outside of the $24$ flex lines to $D$. Over a flex line $L\in G$, $\widehat{D}$ meets the fiber $\widehat{S}_{_L}$ at two points with multiplicities $1$ and $3$, respectively. Let $p_{_L}$ be the point on $\widehat{S}_{_L}$ where $\widehat{D}$ and $\widehat{S}_{_L}$ intersect with multiplicity $3$. By \eqref{LANGLOGAH2E000}, $\widehat{D}$ has an ordinary triple point at $p_{_L}$. That is, $\widehat{D}$ is locally a union of three sections at $p_{_L}$ that meet each other transversely.

Let $\widehat{C} = \nu^{-1}(C)$ be the proper transform of $C$. Suppose that $\widehat{C}$ is a multisection of $\widehat{S}/\widehat{G}$ of degree $\delta$. 
By our hypothesis, the multisection $\widehat{C}\cup \widehat{D}$ has local constant moduli on the fibers of $\widehat{S}/\widehat{G}$. By our previous argument, over every a flex line $L\in G$, $\widehat{C}\cup \widehat{D}$ is locally a union of $\delta + 3$ sections at $p_{_L}$ that meet each other transversely; otherwise, $\theta(C)$ will meet a boundary component of $\overline{U}_{0,4}^{\sym}$.
Therefore,
$$
(\widehat{C} . \widehat{D})_{p_{_L}} = 3\delta.
$$
Thus, we have
$$
\widehat{C} . \widehat{D} = 3\delta \cdot 24 = 72\delta
$$
and hence
$$
\pi_* C . D = 72\delta.
$$
This implies that $\deg \pi_* C = 18\delta$. In addition, we have $\delta \le 3$. Indeed, $\delta = 1$ if $\deg G = 12$, $\delta = 1$ or $2$ if $\deg G = 6$, and $\delta = 1$ or $3$ if $\deg G = 4$. 
\end{proof}

It remains to bound $2g-2+i(C,D)$ for curves $C$ as in Proposition \ref{prop-fibersofphiclasses}. We do this by working more explicitly with the relative tangent sheaf $T_{\theta}$. For a general line $\ell$ in $\PP^2$, $T_{\PP^2}(-\log D)|_{\ell} = \OO \oplus \OO(-1)$ by Lemma \ref{lem-Ominus2splitting}. By \cite{PatelRiedlTseng}, we see that the $\OO$ factor corresponds to deformations of $\ell$ in $\PP^2$ that fix the modulus of $\ell \cap D$, and furthermore, at a point $p \in \ell$, the $\OO$ direction corresponds to the direction that $p$ would move under this deformation, viewing $T_{\PP^2}(-\log D)$ as a subsheaf of $T_{\PP^2}$. Thus, the relative tangent sheaf to $\theta$ will be governed by the sheaf determined by these $\OO$ factors. We now work out the first chern class of this sheaf.

\begin{Proposition}
Let $G_4 = \eta^* \eta_* \pi^*T_{\PP^2}(-\log D)$. Then $G_4$ has first chern class $(0,-7)$ on $V$. Hence, so does $T_\theta$.
\end{Proposition}

\begin{proof}
First observe that $T_\theta$ and $G_4$ are identical subsheaves of $T_V$ outside of the non-defined locus of $\theta$ and the locus of lines with $T_{\PP^2}(-\log D)|_\ell = \OO(1) \oplus \OO(-2)$. By Lemma \ref{lem-Ominus2splitting}, this will be a codimension 2 locus in $V$, so their first chern classes will be the same. Similarly, $R^1 \eta_* \pi^* T_{\PP^2}(-\log D)$ will be 0 outside of a codimension 2 subset, so it won't affect the calculation of $c_1$. 

Using Grothendieck-Riemann-Roch, we see that $\ch(\eta_* \pi^* T_{\PP^2}(-\log D) ) = \eta_*(\ch(\pi^* T_{\PP^2}(-\log D) \td T_{\eta}) .$

Since $T_{\eta} = (-1,2)$, we have $\td(T_{\eta}) = 1- \frac{1}{2}H_2+H_1+\frac{1}{12}(-H_2+2H_1)^2$. Using the defining sequence 
\[ 0 \to T_{\PP^2}(-\log D) \to T_{\PP^2} \to \OO_D(D) \to 0 \]
with the Euler sequence, we see that $\ch(T_{\PP^2}(-\log D)) = 2 - H_1 -\frac{13}{2}H_1^2$.
Combining, we compute $\eta_* ((2-H_1-\frac{13}{2})(1- \frac{1}{2}H_2+H_1+\frac{1}{12}(-H_2+2H_1)^2) = 1-7H + \text{higher terms}$. The result follows.
\end{proof}

\begin{Corollary}\label{cor-finalFibersCor}
Let $C$ be a curve in $V$ in a fiber of the map $\theta$, and suppose that $\pi(C)$ is not a line. Then $\pi(C)$ has degree 18 and $2g(C)-2+i(C,D) \geq 84$. In particular, $2g-2+i \geq \frac{1}{2} \deg C$.
\end{Corollary}
\begin{proof}
This is a simple chern class calculation. Since $C$ is tangent to $T_\theta$, it follows that $T_C(-\log D \cap C)$ maps to $T_{\theta}$, which implies that its degree is at most the degree of $T_{\theta}|_C$. The result follows from the fact that the class of $C$ is $(18,12)$, $(18, 6)$, or $(18,4)$ by Lemma \ref{LANGLOGAH2LEMFIXMODULI}.
\end{proof}

\begin{proof}[Proof of Theorem \ref{thm-varmain}]
We have already completed almost all the necessary arguments. For $n=2$, this follows directly from Theorem \ref{thm-associatedLine} and Corollary \ref{cor-finalFibersCor}. For $n>2$, by Theorem \ref{thm-associatedLine}, we know that the only possible counterexamples to (\ref{eqn-logah3}) lie in the fibers and undefined locus of $\theta$, that is, they form a bounded family of curves, which we call $M$. We can stratify $M$ by both the geometric genus and the intersection type with $D$, so it remains to show that no dominant stratum of curves of $M$ can violate algebraic hyperbolicity. Thus, the result follows from Proposition \ref{prop-algGGLfordominantFamiliesOfCurves}.
\end{proof}

\begin{Proposition} \label{prop-sharpconics}
Let $D$ be a general quartic plane curve in $\PP^2$. Then there exists a conic in $\PP^2$ meeting $C$ set-theoretically at three points. Hence, for $S$ equal to the union of the flex and bitangent lines, the inequality in Theorem \ref{thm-varmain} is sharp.
\end{Proposition}

We remark that it is possible one could get better bounds in Theorem \ref{thm-varmain} by enlarging $S$. We do not know if this is possible.

\begin{proof}[Proof of Proposition \ref{prop-sharpconics}]
We first compute the dimension of the universal family of such pairs. Let $W$ be the variety of pairs $(C,D)$ where $C$ is a smooth conic and $D$ is a quartic plane curve such that $C$ meets $D$ to order 6 at one point and to order 1 at two other points. We wish to compute the dimension of $W$. There is a natural map $\pi_1$ from $W$ to the space $\PP H^0(\OO_{\PP^2}(2))$ of conics, surjective over the locus of smooth conics. We wish to compute the dimension of a fiber of $\pi_1$. Let $C$ be a smooth conic. Then the map $H^0(\PP^2,\OO_{\PP^2}(4)) \to H^0(\OO_C(4))$ is surjective, so the dimension of a general fiber of $\pi_1$ has dimension $\dim H^0(\mathcal{I}_{C}(4)) + \dim C \times \Sym^2 C = 9$. Thus, the dimension of $W$ is 14, the same as the dimension of the space of quartic plane curves. Observe that $W$ will contain pairs $(C,D)$ with $D$ smooth. To see this, observe that the space of possible $D$ meeting $C$ to order at least 6 at a given point $p$ form a linear system. By Bertini, a general $D$ in the linear system will be smooth away from the base locus $p$. However, we see that the reducible conic $C \cup C_1$ will be in the linear system for some conic $C_1$ not containing $p$, so the general element of the linear system will be smooth at $p$ as well. 

We now claim that the map $\pi_2$ from $W$ to the space of quartic plane curves is generically finite. To see this, note that given a fixed $(C,D)$ in $W$ with $D$ smooth, the deformations of $C$ that still meet $D$ to order 6 at one point will be parameterized by $H^0(N_{C/\PP^2}(\log D))$. A simple degree calculation shows that the degree of $N_{C/\PP^2}(\log D)$ is $-1$, and since $C$ and $D$ are both smooth, we see that the normal bundle will be a vector bundle. It follows that $H^0(N_{C/\PP^2}(\log D)) = 0$, which completes the proof.

\end{proof}

\begin{proof}[Proof of Corollary \ref{cor-LangCor}]
We follow Demailly's argument in the projective case.
Suppose that there is a nonconstant map $\alpha: A \to \PP^2 \setminus D$ from a semiabelian variety $A$ to $\PP^2 \setminus D$ whose image does not lie in $S$. Let $[n]:A \to A$ be the map sending $a\to na$ by the group law and let $f:C \hookrightarrow A$ be a curve in $A$ cut out by $\dim A - 1$ generic very ample divisors. 
Clearly, $h_n = \alpha\circ [n]\circ f: C\to \PP^2\setminus D$ is nonconstant for all $n$. Let $C\hookrightarrow \overline{C}$ be the completion of $C$ to a smooth projective curve and $\overline{h}_n: \overline{C}\to \PP^2$ be the completion of $h_n$. Since $\overline{h}_n^{-1}(D)\subset \overline{C} \setminus C$,
$$
2g(\overline{C}) - 2 + |\overline{C} \setminus C| \ge 2g(\overline{C}) - 2 + |\overline{h}_n^{-1}(D)|
$$
Since $\overline{h}_n(C)$ does not lie in $S$, we have
$$
2g(\overline{C}) - 2 + |\overline{C} \setminus C| \ge 2g(\overline{C}) - 2 + |\overline{h}_n^{-1}(D)| \ge \frac{1}2 \deg \overline{h}_n^* L
$$
for $L = \OO_{\PP^2}(1)$ by Theorem \ref{thm-varmain}. This is impossible since $\deg \overline{h}_n^* L\to \infty$ as $n\to \infty$.
\end{proof}

\bibliographystyle{alpha}
\bibliography{xic}

\end{document}